\documentclass[letterpaper,leqno,11pt]{article}

\usepackage{amsmath,amsthm,amsfonts,amssymb}
\usepackage{euscript}
\usepackage{enumitem}
\usepackage[usenames]{color}
\usepackage[authoryear,longnamesfirst]{natbib}
\usepackage{mathtools,comment}
\usepackage{url}
\usepackage{sectsty}
\usepackage[symbol]{footmisc}
\RequirePackage{hyperref}
\usepackage{natbib}
\usepackage{xspace}
\usepackage[height=9in,width=5.75in]{geometry}
\usepackage{graphics,graphicx}

\date{January 12, 2012}

\definecolor{darkblue}{rgb}{0.03,0.03,0.265}
\hypersetup{colorlinks=true,urlcolor=darkblue,citecolor=darkblue,linkcolor=darkblue,%
  pdfauthor={Daniel Remenik},pdftitle={Brunet-Derrida Particle Systems Free and Boundary Equations}}

\mathtoolsset{showonlyrefs,showmanualtags}

\newtheorem{thm}{Theorem}
\newtheorem{lem}{Lemma}[section]
\newtheorem{prop}[lem]{Proposition}

\newtheorem{thm2}[lem]{Theorem}
\theoremstyle{definition}
\newtheorem{rem}[lem]{Remark}
\newtheorem*{rem*}{Remark}

\newcounter{assum}

\newcommand{\pp}{\mathbb{P}}

\newcommand{\ee}{\mathbb{E}}
\newcommand{\rr}{\mathbb{R}}
\newcommand{\nn}{\mathbb{N}}

\newcommand{\cx}{\EuScript{X}}
\newcommand{\cp}{{\cal P}}
\newcommand{\cm}{{\cal M}}

\newcommand{\uno}[1]{\mathbf{1}_{#1}}

\newcommand{\hnu}{\widehat{\nu}}

\newcommand{\hf}{\widehat{f}}

\newcommand{\ls}{\lambda^*}
\newcommand{\ep}{\varepsilon}
\newcommand{\vs}{\vspace{6pt}}

\newcommand{\FB}{\hyperlink{FB}{{\rm(FB)}}\xspace}
\newcommand{\FBp}{{\hyperlink{FBp}{{\rm(FB${}^\prime$)}}}\xspace}
\newcommand{\TW}{\hyperlink{TW}{{\rm(TW)}}\xspace}
\newcommand{\br}[2]{\left<{#1},{#2}\right>}

\sectionfont{\large\selectfont}
\subsectionfont{\normalsize\selectfont}
\setlist{itemsep=-2pt,topsep=2pt,label={\footnotesize\textbullet}}

\numberwithin{equation}{section}

\let\oldsubsec\subsection
\renewcommand{\subsection}[1]{\oldsubsec{#1}}

\let\oldmarginpar\marginpar
\renewcommand\marginpar[1]{\-\oldmarginpar[\raggedleft\footnotesize #1]%
{\raggedright{\small\textsf{#1}}}}

\allowdisplaybreaks[0]

\begin{document}


\title{Brunet-Derrida particle systems, free boundary problems\\
and Wiener-Hopf equations}
\author{Rick Durrett\\
  Duke University\and
  Daniel Remenik\\
  University of Toronto}
\maketitle

\vspace{-10pt}
\footnote[0]{This work was completed while R.D. was a Professor of
  Mathematics and D.R. was a graduate student at Cornell University. Both authors were
  partially supported by NSF grant DMS 0704996 from the probability program.

  \vskip3pt
  ~~(Corrected April 2023)}
\vspace{-20pt}

\begin{abstract}
  We consider a branching-selection system in $\rr$ with $N$ particles which give birth
  independently at rate $1$ and where after each birth the leftmost particle is erased,
  keeping the number of particles constant. We show that, as $N\to\infty$, the empirical
  measure process associated to the system converges in distribution to a deterministic
  measure-valued process whose densities solve a free boundary integro-differential
  equation. We also show that this equation has a unique traveling wave solution traveling
  at speed $c$ or no such solution depending on whether $c\geq a$ or $c<a$, where $a$ is
  the asymptotic speed of the branching random walk obtained by ignoring the removal of
  the leftmost particles in our process. The traveling wave solutions correspond to
  solutions of Wiener-Hopf equations.
\end{abstract}

\section{Introduction and statement of the results}
\label{sec:intro}

We will consider the following branching-selection particle system. At any time $t$ we
have $N$ particles on the real line with positions
$\eta^N_t\!(1)\geq \cdots \geq\eta^N_t\!(N)$. Each one of the $N$ particles gives birth at
rate 1 to a new particle whose position is chosen, relative to the parent particle, using
a given probability distribution $\rho$ on $\rr$. Whenever a new particle is born, we
reorder the $N+1$ particles and erase the leftmost one (so the number of particles is
always kept equal to $N$). We will denote by
$\cx_N=\{(\eta(1),\dotsc,\eta(N))\in\rr^N\!:\eta(1)\geq\dotsm\geq\eta(N)\}$ the state space of
our process.

We learned of this process through the work of \citet{durrMay}, who considered the special
case in which $\rho$ corresponds to a uniform random variable on $[-1,1]$.  However, our
process is a member of a family of processes that first arose in work of \citet{brunDerr},
who studied a discrete analog of the Fisher-Kolmogorov PDE:
\[\frac{\partial h}{\partial t}=\frac{\partial^2 h}{\partial x^2}+h-h^3.\]
In a simpler version of their process, model A in \citet{brunDerr2}, the discrete time
dynamics occur in the following way: at each time step each of the $N$ particles is replaced by a
fixed number $k$ of particles whose displacements from the parent particle are chosen
independently, and then only the $N$ rightmost particles are kept. They conjectured that
the system moves to the right with a deterministic asymptotic speed $v_N$ which increases
as $N\to\infty$ to some explicit maximal speed $v$ at a rate of order $(\log
N)^{-2}$. This slow rate of convergence was recently proved by \citet{berGou} in the case
$k=2$ under some assumptions on the distribution used to choose the locations of the new
particles (one being that new particles are always sent to the right of the parent
particle).

Although we will say something about the behavior of the system for fixed $N$, our main
interest in this paper is to study the behavior of the empirical distribution of the
process as $N\to\infty$. Before proceeding with this, let us specify some
assumptions. When a particle at $x$ gives birth, the new particle is sent to a location
$x+y$ with $y$ being chosen from an absolutely continuous probability distribution
$\rho(y)dy$.  We will assume that $\rho$ is symmetric and that
$\int_{-\infty}^\infty|x|\rho(x)\,dx<\infty$.  The initial condition for our process will always
be specified as follows: each particle starts at a location chosen independently from a
probability measure $f_0(x)dx$, where $f_0(x)=0$ for $x<0$ and $f_0(x)$ is strictly
positive and uniformly continuous\footnote[1]{\emph{(Added April 2023)} ~ The uniform continuity assumption is missing from the published version of this article, where only continuity was specified.} for $x>0$.

\subsection{Convergence to the solution of a free boundary problem}
\label{sec:cvSolFB}

Let
\[\nu^N_t=\frac{1}{N}\sum_{i=1}^N\delta_{\eta^N_t\!(i)}\]
be the empirical measure associated to $\eta^N_t$. Observe that the initial empirical
measure $\nu^N_0\!(dx)$ converges in distribution to $f_0(x)dx$. We will show that, as
$N\to\infty$, this empirical measure process converges to a deterministic measure-valued
process whose densities are the solution of a certain free boundary problem.

Alternatively one could think of the following (weaker) version of the problem.  It is not
hard to see that the probability measure $\ee(\nu^N_t\!(\cdot))$ on $\rr$ is absolutely
continuous. Let $\overline f^N\!(t,x)$ be its density. We want to study its limit as
$N\to\infty$. We expect this limit $f(t,x)$ to correspond to the densities of the limiting
measure-valued process mentioned above. Now observe that if $\xi^N_t$ is a version of our
process in which we do not erase the leftmost particle after births (i.e., $\xi^N_t$ is a
branching random walk), then we would expect that the density of the corresponding
expected empirical measure converges to the solution $\hf(t,x)$ of the following
integro-differential equation:
\begin{equation}
  \frac{\partial\hf}{\partial t}(t,x)=\int_{-\infty}^\infty\!\hf(t,y)\rho(x-y)\,dy.\label{eq:branchLim}
\end{equation}
This is indeed the case, as we will see in Proposition \ref{prop:nuhat}. Observe that the
total mass of $\hf(t,\cdot)$ grows exponentially (in fact
$\int_{-\infty}^{\infty}\hf(t,y)\,dy=e^t$, see \eqref{eq:massHf}).

By adding the selection step to our process we ensure that the limiting density always has
mass 1, but otherwise the branching mechanism is still governed by the convolution term
appearing in \eqref{eq:branchLim}. Thus we expect that if the limit
$f(t,x)=\lim_{N\to\infty}\overline f^N\!(t,x)$ exists, then it has to satisfy the
following: there exists a continuous increasing function
$\gamma\!:[0,\infty)\longrightarrow\rr$ with $\gamma(0)=0$ such that $(f(t,x),\gamma(t))$
is the unique solution to the following free boundary problem \hypertarget{FB}{{\rm(FB)}}:
\begin{gather}
  \frac{\partial f}{\partial t}(t,x)=\int_{-\infty}^\infty\!f(t,y)\rho(x-y)\,dy
  \quad\forall\,x>\gamma(t)\tag{FB1}\label{eq:FB1}\\
  \int_{\gamma(t)}^\infty\!f(t,y)\,dy=1\tag{FB2}\label{eq:FB2}\\
  f(t,x)=0\quad\forall\,x\leq\gamma(t)\tag{FB3}\label{eq:FB3}
\end{gather}
with initial condition $f(0,x)=f_0(x)$ for all $x\in\rr$. $\gamma(t)$ is a moving boundary
which keeps the mass of $f(t,\cdot)$ at 1, but the speed at which it moves is not known in
advance and depends in turn on $f$.

\begin{rem} \emph{(Added April 2023)} ~  
It is not a priori obvious that \FB has a solution, let alone that such a solution is
unique. 
Existence is proved below as a consequence of the arguments we will use to prove the existence of the limiting density.
In the published version of this article a proof of uniqueness was also provided, but it turns out that the argument was flawed, see Remark \ref{rem:flaweduniq} below.
Uniqueness has now been proved in a more general setting in \cite{atar}.
In that paper, the author in fact proves a version of Theorem \ref{thm:fb} of this paper for a more general, multidimensional variant of $\nu^N_t$.
\end{rem}

We will denote by $\cp$ the space of
probability measures on $\rr$, which we endow with the topology of weak convergence, and
by $D([0,T],\cp)$ the space of c\`adl\`ag functions from $[0,T]$ to $\cp$ endowed with the
Skorohod topology.

\begin{thm}\label{thm:fb}
  For any fixed $T>0$ the sequence of $\cp$-valued processes $\nu^N_t$ on $[0,T]$
  converges in distribution in $D([0,T],\cp)$ to a deterministic $\nu_t\in
  D([0,T],\cp)$. $\nu_t$ is absolutely continuous with respect to the Lebesgue measure for
  every $t\in[0,T]$ and the corresponding densities $f(t,\cdot)$ are characterized by the
  following: there exists a continuous, strictly increasing function
  $\gamma\!:[0,\infty)\longrightarrow\rr$ with $\gamma(0)=0$ such that
  $(f(t,x),\gamma(t))$ is the unique solution of the free boundary problem \FB. In
  particular for $x>\gamma(t)$, $f(t,x)$ is strictly positive, jointly continuous in $t$ and
  $x$ and differentiable in $t$.
\end{thm}

Let us remark that there are at least three other examples in the literature of particle
systems converging to the solution of a free boundary equation, but in those cases the
limiting equation is of a different type. \citet{chaySwin} study a particle system
corresponding to a liquid-solid system with an interface subject to melting and freezing,
\citet{landimOllaVolchan} study a tracer particle moving in a varying environment
corresponding to the simple symmetric exclusion process, and \citet{gravQuas} study an
internal diffusion limited aggregation model. In all three cases cases an hydrodynamic limit is
proved with the limiting equation being closely related to the famous Stefan problem, which
involves free boundary problems for the heat equation where the moving boundary separates
the solid and liquid phases (see \citet{meirmanov} and references therein for more on this
problem).

\subsection{Behavior of the finite system}\label{sec:finite}

To study the finite system it will be useful to introduce the \emph{shifted} process
$\Delta^N_t$, which we define as follows:
\[\quad\quad\Delta^N_t=(\Delta^N_t\!(1),\dotsc,\Delta^N_t\!(N))\qquad\text{with }\quad
\Delta^N_t\!(j)=\eta^N_t\!(j)-\eta^N_t\!(N).\] Observe that $\Delta^N_t\!(N)$ is always 0.
It is clear that $\Delta^N_t$ is also a Markov process, and its transitions are the same
as those of $\eta^N_t$ except that after erasing the leftmost particle the $N$ remaining
particles are shifted to the left so that the new leftmost one lies at the origin.

We will denote by $\min\eta^N_t=\eta^N_t\!(N)$ and $\max\eta^N_t=\eta^N_t\!(1)$ the locations
of the leftmost and rightmost particles in $\eta^N_t$. 

\begin{thm}\label{thm:finite}
  For every fixed $N>0$ the following hold:
  \begin{enumerate}[label=(\alph*)]
  \item There is an $a_N>0$ such that
    \[\lim_{t\to\infty}\frac{\min\eta^N_t}{t}=\lim_{t\to\infty}\frac{\max\eta^N_t}{t}=a_N\]
    almost surely and in $L^1$. Moreover, the sequence $(a_N)_{N>0}$ is non-decreasing.
  \item The process $\Delta^N_t$ has a unique stationary distribution $\mu_N$, which is
    absolutely continuous.
  \item For any (random or deterministic) initial condition $\nu_0$ we have
    \[\left\|\pp^{\nu_0}\!\left(\Delta^N_t\in\cdot\right)-\mu_N(\cdot)\right\|_{TV}\xrightarrow[t\to\infty]{}0.\]
  \end{enumerate}
\end{thm}

From this point on we will assume that the displacement distribution $\rho$ has
exponential decay. To be precise, we assume that there is an $\alpha>0$ such that
\begin{equation}
  \label{eq:expDecayRho}\rho(x)\leq Ce^{-\alpha|x|}.
\end{equation}
for some $C>0$. We will write
\[\Theta=\sup\left\{\alpha>0\!:\,\sup_{x\in\rr}\big[e^{\alpha|x|}\rho(x)\big]<\infty\right\}.\]
That is, $\Theta\in(0,\infty]$ is the maximal exponential rate of decay of $\rho$ in the
sense that $\rho(x)\leq Ce^{-\alpha x}$ for some $C>0$ when $\alpha<\Theta$ but not when
$\alpha>\Theta$. $\Theta$ may be $\infty$ (as in the cases where $\rho$ has compact support
or $\rho$ corresponds to a normal distribution), while $\Theta>0$ is ensured by
\eqref{eq:expDecayRho}.

Now let
\begin{equation}
  \phi(\theta)=\int_{-\infty}^\infty e^{\theta x}\rho(x)\,dx\label{eq:defphi}
\end{equation}
be the moment generating function of the displacement distribution $\rho$.  Equation
\eqref{eq:expDecayRho} and our definition of $\Theta$ imply that $\phi(\theta)<\infty$ for
$\theta\in(-\Theta,\Theta)$.  To avoid unnecesary technical complications we will make the
following extra assumption, which in particular implies that $\phi(\theta)=\infty$ for
$|\theta|>\Theta$:
\begin{equation}
  \label{eq:expMoments:2}
  \frac{\phi(\theta)}{\theta}\xrightarrow[\theta\to\Theta^-]{}\infty.
\end{equation}
This assumption always holds when $\Theta=\infty$: choosing $0<l_1<l_2$ so that
$\rho(x)\geq M$ for some $M>0$ and all $x\in[l_1,l_2]$ we get
\[\frac{1}{\theta}\int_{-\infty}^\infty\!e^{\theta x}\rho(x)\,dx
\geq\frac{M(l_2-l_1)}{\theta}e^{\theta l_1}\xrightarrow[\theta\to\infty]{}\infty.\]

Our next result will relate the asymptotic propagation speed $a_N$ of our process
$\eta^N_t$ with the asympotic speed of the rightmost particle in the branching random walk
$\xi^N_t$.

\begin{thm}\label{thm:limAN}
\[\lim_{N\to\infty}a_N=a,\]
where $a$ is the asymptotic speed of the rightmost particle in $\xi^N_t$ , i.e., in a
branching random walk where particles branch at rate 1 and their offspring are displaced
by an amount chosen according to $\rho$.
\end{thm}

\begin{rem}
  If $\rho$ is uniform on $[-1,1]$ this follows from (8) in \citet{durrMay}. However, the
  couplings on which the proof is based extend easily to our more general setting, so we
  do not give the details of the proof.
\end{rem}

The speed $a$ has an explicit expression (see \citet{bigg}): by standard results of the
theory of large deviations, if $S_t$ is a continuous time random walk jumping at rate 1
and with jump distribution $\rho$, then the limit
\begin{equation}
  \Lambda(x)=\lim_{t\to\infty}\frac{1}{t}\log\pp\!\left(S_t>xt\right)\label{eq:defLambda}
\end{equation}
exists and equals $-\left(\sup_{\theta>0}\{x\theta-\phi(\theta)\}+1\right)$; $a$ is given
then by the formula $\Lambda(a)=-1$ (see \citet{durrMay} for more on
this).

\subsection{Traveling wave solutions}\label{sec:wave}
A \emph{traveling wave solution} of \FB is a solution of the form $f(t,x)=w(x-ct)$ and
$\gamma(t)=ct$ for some $c>0$ (with initial condition $f_0(x)=w(x)$).

If $w$ is a traveling wave solution, then from \eqref{eq:FB1} we get
\[cw'(z)=-\int_0^\infty\!w(y)\rho(z-y)\,dy\quad\forall\,z>0,\] so integrating from $z=x$
to $\infty$ we deduce that $w$ must solve the following system of equations
\hypertarget{TW}{{\rm(TW)}}:
\begin{gather}
  w(x)=\frac{1}{c}\int_0^\infty\!w(y)R(x-y)\,dy\quad\forall\,x>0,\tag{TW1}\label{eq:waveSol:a}\\
  \int_0^\infty\!w(x)\,dx=1,\tag{TW2}\label{eq:waveSol:b}\\
  w(x)=0\quad\forall\,x\leq0,\tag{TW3}\label{eq:waveSol:c}\\
  w(x)\geq0\quad\forall\,x>0,\tag{TW4}\label{eq:waveSol:d}
\end{gather}
where
\[R(x)=\int_x^\infty\!\rho(y)\,dy\] is the tail distribution of $\rho$. On the other hand,
it is easy to check that if $w$ satisfies \TW and $f_0(x)=w(x)$, then
$(w(x-ct),ct)$ is the solution of \FB.

Equation \eqref{eq:waveSol:a} is known as a Wiener-Hopf equation. Equations of this type have been
studied since at least the 1920's (at the time in relation with the theory of radiative
energy transfer), and have since been extensively studied and found relevance in diverse
problems in mathematical physics and probability.  In general, these equations can be
solved using the Wiener-Hopf method, which was introduced in \citet{wienHopf}
(see Chapter 4 of \citet{paleyWien} and also \citet{krein}).
But the solutions provided by this method are not necessarily positive, so they are not
useful in our setting. Instead, we will rely on the results of \citet{spitz1}, 
who studied these equations via probabilistic methods in the case where $R(x)/c$ is a
probability kernel.

To do that we need to convert our equation to one where the kernel with respect to which
we integrate is a probability kernel. To that end, we need to make the following
observation. In Lemma \ref{lem:goodSpeed} we will show that there is a $\ls\in(0,\Theta)$
such that
\begin{equation}
  \frac{\phi(\ls)}{\ls}=\min_{\lambda\in(0,\Theta)}\frac{\phi(\lambda)}{\lambda}=a,\label{eq:claimGoodSpeed}
\end{equation}
where $\phi$ is the moment generating function of $\rho$ defined in \eqref{eq:defphi} and
$a$ is the asymptotic speed introduced in Theorem \ref{thm:limAN}. On the other hand, the
function $\lambda\mapsto\phi(\lambda)/\lambda$ is continuous and goes to $\infty$ as
$\lambda\to0$ (and moreover, as we will show in Lemma \ref{lem:goodSpeed}, it is
decreasing on $(0,\ls)$). Thus for every $c\geq a$ there is a $\lambda\in(0,\ls]$ such
that $\phi(\lambda)/\lambda=c$.

Observe that the tail distribution $R$ of $\rho$ has the same decay as $\rho$ (see
 \eqref{eq:expDecayRho}). That is, for every $0<\alpha<\Theta$ there is a $C>0$ so that
 \begin{equation}
   \label{eq:expDecayR}
   R(x)\leq Ce^{-\alpha x}\qquad\forall\,x\geq0.
 \end{equation}
 Fix $c\geq a$ and use the above observation to pick a $\lambda\in(0,\ls]$ such that
 $\phi(\lambda)/\lambda=c$. Equation \eqref{eq:expDecayR} implies that the function $x\mapsto
 e^{\lambda x}R(x)$ is in $L^1(\rr)$. Moreover, integration by parts yields
\begin{equation}
\int_{-\infty}^\infty\!e^{\lambda x}R(x)\,dx=\frac{\phi(\lambda)}{\lambda}.\label{eq:goodMass}
\end{equation}
Therefore
\begin{equation}
  k(x)=\frac{\lambda}{\phi(\lambda)}e^{\lambda
    x}R(x)\label{eq:defK}
\end{equation}
is a probability kernel. On the other hand, if $w$ is a solution of \TW with
$c=\phi(\lambda)/\lambda$ then it is easy to check that $u(x)=e^{\lambda x}w(x)$ satisfies
\begin{equation}
  u(x)=\int_0^\infty\!u(y)k(x-y)\,dy\quad\forall\,x\geq0.\label{eq:waveSol:densU}
\end{equation}
Thus the idea will be to recover solutions of \TW from positive solutions
of \eqref{eq:waveSol:densU}.

Positive solutions of \eqref{eq:waveSol:densU} can be regarded as densities of stationary
measures for the following Markov chain. Let $\xi_n$ be a sequence of i.i.d. random
variables with distribution given by $k$, let $X_0=0$ and define
\[X_{n+1}=\left(X_n+\xi_n\right)^+,\] where $x^+=\max\{0,x\}$. This chain appears, for
example, in the study of ladder variables for a random walk (see Chapter XII of
\citet{feller}) and in the study of the GI/G/1 queue (see Chapter 5 of
\citet{durrPTE}). If $u$ satisfies \eqref{eq:waveSol:densU}, it is 0 on the negative
half-line and it is non-negative on the positive half-line, then the measure (supported on $[0,\infty)$)
having $u$ as its density is invariant for $X_n$. Assuming that $\ee(|\xi_1|)<\infty$,
$X_n$ is recurrent, null-recurrent or transient according to whether $\ee(\xi_1)$ is
negative, zero or positive. As we will see in Section \ref{sec:proof:wave}, this
expectation is negative in our case if and only if $\lambda<\ls$ and it is zero for
$\lambda=\ls$. In both cases the theory of recurrent Harris chains suggests (and Theorem
\ref{thm:spitzer} will prove) that there exists a unique (up to multiplicative constant)
invariant measure for $X_n$, although this measure may not be finite in the null-recurrent
case. The difference between the recurrent and null-recurrent cases explains the
difference between the cases $c>a$ and $c=a$ in Theorem \ref{thm:wave} below. The fact
that the chain is transient when $\lambda>\ls$ suggests that there are no positive
solutions of \eqref{eq:waveSol:densU} for these values of $\lambda$ (again see
Theorem \ref{thm:spitzer} for a proof). This in turn hints at the possibility that there
are no solutions of \TW for $c<a$. Our proof of this fact will not rely in
seeing \TW as a Wiener-Hopf equation, but instead will use explicitly the
fact that its solutions are traveling wave solutions of \FB.

\begin{thm}\label{thm:wave}
  Assume that \eqref{eq:expDecayRho} and \eqref{eq:expMoments:2} hold.
  \begin{enumerate}[label=(\alph*)]
  \item If $c\geq a$ the equation \TW has a unique solution $w$. This
    solution is differentiable except at the origin.
  \item When $c>a$, and letting $\lambda\in(0,\ls)$ be such that
    $\phi(\lambda)/\lambda=c$, the solution satisfies $\int_0^\infty e^{\lambda
      x}w(x)\,dx<\infty$ (which, in particular, implies that $\int_x^\infty
    w(y)\,dy=o(e^{-\lambda x})$). Moreover, if $\widetilde\lambda>\lambda$ then
    $\sup_{x}e^{\widetilde\lambda x}\int_x^\infty w(y)\,dy=\infty$.
  \item When $c=a$ the solution satisfies $\int_x^\infty
    w(y)\,dy=O(e^{-\ls x})$ and $\int_0^xe^{\ls y}w(y)\,dy=O(x)$. The last integral
    goes to $\infty$ as $x\to\infty$.
  \item If $c<a$ the equation \TW has no solution.
  \end{enumerate}
\end{thm}

We remark that, when $c>a$, the solution $w$ given by the theorem can be obtained by the
following limiting procedure.  Take $0<\lambda<\ls$ as in the above statement and let
$u_0$ be the density of any non-negative random variable whose distribution is absolutely
continuous.  Now let $w_0(x)=e^{-\lambda x}u_0(x)$ and then for $n\geq1$ let
\[w_{n+1}(x)=\frac{1}{c}\int_0^\infty\!w_n(y)R(x-y)\,dy\quad\text{for }x\geq0.\] Then the
limit $w_\infty(x)=\lim_{n\to\infty}w_n(x)$ exists and defines an integrable continuous
function. The solution $w$ is then given by $w(x)=Kw_\infty(x)$ with $K>0$ chosen so that
$w$ integrates to 1.  The fact that $w$ has this representation follows from the results
of Spitzer (see Theorem \ref{thm:spitzer}).

\vs

The rest of the paper is devoted to proofs, with one section devoted to each one of the
proofs of Theorems \ref{thm:fb}, \ref{thm:finite} and \ref{thm:wave}.

\section{Proof of Theorem \ref{thm:fb}}
\label{sec:proof:fb}

\subsection{Outline of the proof}

Most of the work in the proof of Theorem \ref{thm:fb} will correspond to showing that for
each fixed $t\geq0$ the tail distribution of our process at that time, defined as
\[F^N(t,x)=\nu^N_t([x,\infty)),\] converges (almost surely and in $L^1$) to a
deterministic limit $F(t,x)$ corresponding to the tail distribution of a random variable
and that, moreover, the limit $F(t,x)$ has a density $f(t,x)$ (that is
$F(t,x)=\int_x^\infty f(t,y)\,dy$) which solves \FB.

To achieve this we will compare the process $\nu^N_t$ with two auxiliary
measure-valued processes $\nu^{N,k}_t$ and $\nu^k_t$. As we will see below, the first of
the two will be a stochastic process, but the second one will be deterministic.

\begin{rem}
  To avoid confusion (and notational complications) we will use the following convention:
  upper-case superscripts, such as in $\nu^N_t$ or $F^N(t,x)$, refer to quantities
  associated with our stochastic process $\eta^N_t$, while lower-case superscripts, such
  as in $\nu^k_t$ or the function $F^k(t,x)$ which we will introduce below, refer to deterministic
  quantities associated with the deterministic process $\nu^k_t$.
\end{rem}

We begin by defining a process $\eta^{N,k}_t$ with values in $\cup_{M\geq N}\cx_M$
inductively as follows. For each $m=0,\dotsc,2^k-1$, run the process with no killing on
the interval $[\frac{m}{2^k},\frac{m+1}{2^k})$ and then at time $\frac{m+1}{2^k}$
repeatedly delete the leftmost particle until there are only $N$ left. To make clear the
distinction between the particles we erase in $\eta^N_t$ and those we erase in this
modified process at dyadic times, we will refer to this last procedure as \emph{shaving
  off the extra mass} in $\eta^{N,k}_t$.

Having defined $\eta^{N,k}_t$ we now define $\nu^{N,k}_t$ as the empirical measure
associated to it:
\[\nu^{N,k}_t=\frac{1}{N}\sum_{i=1}^{|\eta^{N,k}_t|}\delta_{\eta^{N,k}_t(i)},\]
where $|\eta^{N,k}_t|\geq N$ is the number of particles in $\eta^{N,k}_t$. In everything
that follows we will consider the c\`adl\`ag version of $\eta^{N,k}_t$ and $\nu^{N,k}_t$,
and we do the same for the other processes and functions defined below.

The first step in the proof of Theorem \ref{thm:fb} will be to study the convergence of
the tail distribution of $\nu^{N,k}_t$, defined by
\[F^{N,k}(t,x)=\nu^{N,k}_t([x,\infty)).\]
We will see in Proposition \ref{prop:cvNuNk:N} that $F^{N,k}(t,x)$ converges in
probability to $F^N(t,x)$ as $k\to\infty$, and a key fact will be that this convergence is
uniform in $N$.

The second auxiliary process, $\nu^k_t$, will turn out to be the limit of $\nu^{N,k}_t$ as
$N\to\infty$. We define it in terms of its density, which is constructed inductively on
each of the dyadic subintervals of $[0,1]$. We let $f^k(0,x)=f_0(x)$. If we have
constructed $f^k$ up to time $\frac{m}{2^k}$ for some $m\in\{0,\dotsc,2^k-1\}$ then for
$t\in(\frac{m}{2^k},\frac{m+1}{2^k})$ we let $f^k(t,x)$ be the solution of
\begin{equation}
  \frac{\partial f^k}{\partial t}(t,x)=\int_{-\infty}^\infty\!f^k(t,y)\rho(x-y)\,dy.\label{eq:fK}
\end{equation}
Then at time $\frac{m+1}{2^k}$ we let $X^k_{m+1}$ be such that
\begin{equation}
  \int_{X^k_{m+1}}^\infty\!f^k\!\left(\left(\tfrac{m+1}{2^k}\right)\!-,y\right)dy=1\label{eq:defXmk}
\end{equation}
and define
\[f^k\!\left(\tfrac{m+1}{2^k},x\right)
=f^k\!\left(\left(\tfrac{m+1}{2^k}\right)\!-,x\right)\uno{x>X^k_{m+1}}.\] In other words,
on each dyadic subinterval we let $f^k$ evolve following \eqref{eq:branchLim} and then at
each dyadic time we shave off the extra mass in $f^k$. The measure $\nu^k_t$ is defined as
the measure having $f^k(t,\cdot)$ as its density. We also denote by $F^k$ be the tail
distribution of $\nu^k_t$:
\[F^k(t,x)=\nu^k_t([x,\infty))=\int_x^\infty\!f^k(y)\,dy.\]

We will show in Lemma \ref{lem:FkIncr} that, for fixed $t$ and $x$, $F^k(t,x)$ is
decreasing in $k$, so we may define $F(t,x)=\lim_{k\to\infty}F^k(t,x)$. We will show in
Proposition \ref{prop:cvNuNk:N} that $F^{N,k}(t,x)$ converges in probability to $F^k(t,x)$
as $N\to\infty$ for fixed $k$. Since the convergence of $F^{N,k}(t,x)$ to $F^N(t,x)$ as
$k\to\infty$ is uniform in $N$, we will be able to interchange limits to obtain the
convergence of $F^N(t,x)$ to $F(t,x)$ as $N\to\infty$. The rest of the proof will consist
of showing that $F(t,x)$ has a density which solves \FB and then to extend the convergence
of the tail distributions $F^N(t,x)$ for fixed $t$ and $x$ to the convergence of the
measure-valued process $\nu^{N,k}_t$.

\subsection{Convergence of the auxiliary processes}

Our first result will allow us to deduce the limiting behavior of $F^{N,k}(t,x)$ as
$N\to\infty$ inside each dyadic subinterval. Recall that in Section \ref{sec:cvSolFB} we
introduced the branching random walk $\xi^N_t$ defined like $\eta^N_t$ but with no
killing. Let $\hnu^N_t$ be the associated empirical measure. Let $\cm$ be the space of
finite measures on $\rr$, endowed with the topology of weak convergence, and let
$C([0,1],\cm)$ and $D([0,1],\cm)$ be the spaces of continuous and c\`adl\`ag functions
from $[0,1]$ to $\cm$ endowed respectively with the uniform and Skorohod topologies.

\begin{prop}\label{prop:nuhat}
  The empirical process $\hnu^N_t$ associated to the branching random walk $\xi^N_t$
  converges in distribution in $D([0,1],\cm)$ to a deterministic $\hnu_t$ in
  $C([0,1],\cm)$ which for each $t\in[0,1]$ is absolutely continuous with respect to the
  Lebesgue measure. If we denote the density of $\hnu_t$ by $\hf(t,x)$ then $\hf(t,x)$ is
  the unique solution to the integro-differential equation
  \begin{equation}
    \frac{\partial\hf}{\partial t}(t,x)=\int_{-\infty}^\infty\!\hf(t,y)\rho(x-y)\,dy\label{eq:eqHat}
  \end{equation}
  on $[0,1]$ with initial condition $\hf(0,x)=f_0(x)$.
\end{prop}

\begin{proof}
  By Theorem 5.3 of \citet{fourMel} we have that $\hnu^N_t$ converges in distribution to a
  deterministic $\hnu_t$ in $C([0,1],\cm)$ which is the unique solution of the following
  system: for all bounded and measurable $\varphi$,
  \begin{equation}
    \int_{-\infty}^\infty\!\varphi(x)\,\hnu_t(dx)=\int_{-\infty}^\infty\!\varphi(x)\,\hnu_0(dx)+\int_0^t\!\int_{-\infty}^\infty
    \!\int_{-\infty}^\infty\!\varphi(y)\rho(x-y)\,dy\,\hnu_s(dx)\,ds.\label{eq:eqNuHat}
  \end{equation}
  Moreover, by Proposition 5.4 of the same paper, $\hnu_t$ is absolutely continuous for
  all $t\in[0,1]$, and hence its density $\hf(t,x)$ must satisfy $\hf(0,x)=f_0(x)$ and
  \[\frac{d}{dt}\int_{-\infty}^\infty\!\varphi(x)\hf(t,x)\,dx
  =\int_{-\infty}^\infty\!\int_{-\infty}^{\infty}\!\varphi(y)\rho(x-y)\hf(t,x)\,dy\,dx.\]
  Taking $\varphi=\uno{[z,z+h]}$, dividing by $h$, taking $h\to0$ and using the symmetry
  of $\rho$ we deduce that $\hf$ satisfies \eqref{eq:FB1} at $z$ and the result follows.
\end{proof}

Recall that $f_0(x)>0$ for $x\geq0$. It is clear that if $\hf$ solves \eqref{eq:eqHat} then
$\hf(t,x)>0$ for all $t\in[0,1]$ and $x\in\rr$.

\begin{prop}\label{prop:cvNuNk:N}
  For every fixed $k\geq1$ and every $t\in[0,1]$ and $x\in\rr$, $F^{N,k}(t,x)$ converges
  in probability to $F^k(t,x)$.
\end{prop}

\begin{proof}
  Proposition 2.1 implies that $\nu^{N,k}_t([x,\infty))$ converges in probability to
  $\nu^k_t([x,\infty))$ for all $t \in [0,\frac{1}{2^k})$ and all $x\in\rr$. Using the
  partial order in $\cm$ given by $\mu\preceq\nu$ if and only if
  $\mu([x,\infty))\leq\nu([x,\infty))$ for all $x\in\rr$, it is clear that the mappings $t
  \mapsto\nu^{k}_t$ and $t\mapsto \nu^{N,k}_t$ are increasing on $[0,\frac{1}{2^k})$, and
  thus the limits $\lim_{t\uparrow 1/2^k} \nu^{k}_t = \nu^{k}_{-}$ and $\lim_{t\uparrow
    1/2^k} \nu^{N,k}_t = \nu^{N,k}_{-}$ exist (for $\nu^{N,k}_t$ this statement holds
  almost surely). On the interval $[0,\frac{1}{2^k})$ the process $\nu^k_t$ is the same as
  the process $\hnu_t$ defined in Proposition \ref{prop:nuhat}. On the
  other hand, by \eqref{eq:eqHat} we have
  \[\int_{-\infty}^\infty\!\hf(t,x)\,dx=1+\int_0^t\!\int_{-\infty}^\infty\!\hf(s,x)\,dx\,ds,\]
  whence it is easy to see that
  \begin{equation}
    \int_{-\infty}^\infty\hf(t,x)\,dx=e^t.\label{eq:massHf}
  \end{equation}
  Therefore for $0\leq s\leq
  t\leq\frac{1}{2^k}$ we have
  \[(\nu^k_t - \nu^k_s)(\rr)=e^t - e^s.\]
  Similarly, $\nu^{N,k}_t$ corresponds to the
  branching random walk $\xi^N_t$ on $[0,\frac{1}{2^k})$, and thus using (5.4) of
  \citet{fourMel} we can see that $\ee(\nu^{N.k}_t(\cdot))$ equals $\hnu_t(\cdot)$ on this
  interval, so for $0\leq s\leq t\leq\frac{1}{2^k}$ we have
  \[\ee\!\left((\nu^{N,k}_t - \nu^{N,k}_s)(\rr)\right) = e^t - e^s.\]
  Using these two equalities together with the fact that $\nu^k_{-}$ is absolutely
  continuous it is easy to see that $\nu^{N,k}_{-}([x,\infty))$ converges in probability
  to $\nu^{k}_{-}([x,\infty))$ for all $x$. Since $x \mapsto \nu^{k}_{-}((x,\infty))$ is
  strictly decreasing at $x=X^k_1$, the location of the point at which the mass of
  $\nu^{N,k}_t$ is shaved off at time $\frac{1}{2^k}$ converges in probability to
  $X^k_1$. The result for $t=\frac{1}{2^k}$ follows from this, and induction gives the
  desired result.
\end{proof}  

Next we turn to the convergence of $F^{N,k}(t,x)$ as $k\to\infty$.

\begin{prop}\label{prop:cvNuNk:k}
  For every given $t\in[0,1]$ and $x\in\rr$, $F^{N,k}(t,x)$ converges in probability to
  $F^N\!(t,x)$ as $k\to\infty$, uniformly in $N$.
\end{prop}

The proof depends on the following lemma:

\begin{lem}\label{lem:cvMassNuNk}
  We can couple $\eta^{N,k}_t$ and $\eta^N_t$ (starting with the same initial
  configuration) in such a way that the following holds: for every $N\geq1$, $k\geq1$ and $t\in[0,1]$,
  \[\ee\big(|\eta^N_{t}\Delta\eta^{N,k}_{t}|\big)
  \leq N\frac{e-1}{e^{2^{-k}}-1}2^{-2k+1}e^{2^{-k}}+N2^{-k+2},\] where $A\Delta B=A\!\setminus\!B\cup
  B\!\setminus\!A$.
\end{lem} 

Before proving the lemma we need to give an explicit construction of the process
$\eta^N_t$. Consider an i.i.d. family $(U^N_i)_{i\geq1}$ with uniform distribution on
$\{1,\dotsc,N\}$ and an i.i.d. family $(R_i)_{i\geq1}$ with distribution $\rho$ and let
$(T^N_i)_{i\geq1}$ be the jump times of a Poisson process with rate $N$. To construct
$\eta^N_t$ we proceed as follows: at each time $T^N_i$ we branch $\eta^N_{T^N_i}(U^N_i)$
using $R_i$ for the displacement, erase the leftmost particle, and then relabel the
particles to keep the ordering. The reader can check that the resulting process $\eta^N_t$
has the desired distribution.

\begin{proof}
  The coupling will be constructed inductively on each dyadic subinterval of $[0,1]$.  We
  start both processes with the same initial configuration. The idea will be to use the
  same branching times and displacements whenever possible. To do this we will decompose
  $\eta^{N,k}_t$ in the following way (for convenience we regard $\eta^{N,k}_t$ and
  $\eta^N_t$ here as sets)
  \begin{equation}
    \eta^{N,k}_t=G^{N,k}_t\cup D^{N,k}_t\cup B^{N,k}_t,\label{eq:goodDangBad}
  \end{equation}
  where the unions are disjoint and:
  \begin{itemize}[itemsep=2pt]
    \item $G^{N,k}_t\subseteq\eta^N_t$ are ``good
      particles'', i.e. particles which are coupled, in the sense that $G^{N,k}_t=\eta^{N,k}_t\cap\eta^N_t$;
    \item $B^{N,k}_t$ are ``bad particles'', i.e. particles which are not coupled;
    \item $D^{N,k}_t$ are ``dangerous particles'', i.e. particles which will become bad if not erased at
      the next dyadic time.
  \end{itemize}

  The basic idea of our coupling is the following. Good particles, which are present in both processes, evolve together using the same branching times and
  locations. When a good particle branches, a particle is erased from $\eta^{N}_t$ but not
  from $\eta^{N,k}_t$. If the particle erased from $\eta^N_t$ is not a good particle then
  the coupling is not affected. Otherwise, if the erased particle is good, we relabel it
  as dangerous in $\eta^{N,k}_t$. Observe that if this particle does not branch before the
  next dyadic time then it will not affect the coupling since it will surely get erased
  (by definition it is to the left of every good particle). When dangerous or bad
  particles give birth in $\eta^{N,k}_t$ we label their offspring as bad. Our goal will be
  to bound the number of bad particles.

  Now we define the coupling more precisely. The first step is to construct $\eta^N_t$
  using the sequences $U^N_i$, $R_i$ and $T^N_i$ as described in the paragraph preceding
  this proof. Now we need to explain how to construct $\eta^{N,k}_t$ and decompose it into
  good, dangerous and bad particles. For the initial condition we choose
  $G^{N,k}_0=\eta^N_0$ and $D^{N,k}_0=B^{N,k}_0=\emptyset$.

  We assume that we have constructed the coupling until time $\frac{m}{2^k}$ for some
  $0\leq m\leq 2^k-1$ and that the following holds:
  \begin{equation}
    \label{eq:condCoupl}
    G^{N,k}_{m/2^k}=\eta^{N,k}_{m/2^k}\cap\eta^{N}_{m/2^k}\quad
    \text{and}\quad D^{N,k}_{m/2^k}=\emptyset.
  \end{equation}
  Observe that this condition holds trivially for $m=0$. Observe also that
  $\eta^N_{m/2^k}\setminus G^{N,k}_{m/2^k}$ and $B^{N,k}_{m/2^k}$ both have
  $N-|G^{N,k}_{m/2^k}|$ particles, and thus we may identify particles in each set in a
  one-to-one fashion by, for example, going from left to right in each set. 

  Next we define the coupling on the interval $(\frac{m}{2^k},\frac{m+1}{2^k}]$. Let
  $\frac{m}{2^k}\leq T^N_{I_m}<T^N_{I_m+1}<\dotsm\leq T^N_{J_m}\leq\frac{m+1}{2^k}$ be the
  sequence of branching times for particles in $\eta^N_t$ on this time interval (there are
  almost surely a finite number $J_m-I_m+1$ of such times). We remark that after each
  branching event we will still have each particle in $\eta^N_t\setminus G^{N,k}_t$
  identified with one particle in $B^{N,k}_t$ (see the second and last bullet below). For
  each $I_m\leq i\leq J_m$ we do the following:
  \begin{itemize}[itemsep=2pt]
    
  \item If the branching at time $T^N_i$ occurs at a particle which is in
    $G^{N,k}_{T^N_i}$, we add the new particle to $G^{N,k}_{T^N_i}$.
    
  \item Otherwise, if the particle that is undergoing a branching in $\eta^N_t$ at time $T^N_i$ is
    not a good particle (and therefore it is not in $\eta^{N,k}_{T^N_i}$), we use the
    branching time and displacement to branch the particle in $B^{N,k}_{m/2^k}$ which is
    identified with it, and we identify this new bad particle with the new particle born
    in $\eta^N_t$ at this branching event.
    
  \item If the particle erased from $\eta^N_{T^N_i}$ after the branching at
    time $T^N_i$ is good (i.e. it is also in $G^{N,k}_{T^N_i}$), we relabel it as dangerous by
    moving it from $G^{N,k}_{T^N_i}$ to $D^{N,k}_{T^N_i}$. This dangerous particle in
    $\eta^{N,k}_t$ will not have an associated particle in $\eta^N_t$, so we use
    independent branching times and displacements for it and all its offspring, and label
    all its offspring as bad.
    
  \item Otherwise, if the particle erased from $\eta^N_{T^N_i}$ after the branching at
    time $T^N_i$ is not in $G^{N,k}_{T^N_i}$ then there is a particle in $B^{N,k}_{T^N_i}$
    identified with it; after the branching we remove this identification and use
    independent branching times and displacements for this bad particle and its offspring.
    
  \end{itemize}
  We remark that the offspring of dangerous and bad particles in $\eta^{N,k}_t$ is always
  labeled as bad and that whenever one of these particles has no associated particle in
  $\eta^N_t$ it uses independent branching times and displacements.
  
  The rules used to identify particles in $\eta^N_t\setminus G^{N,k}_t$ with particles in
  $B^{N,k}_t$ are not particularly important, the main point is that every branching event
  in $\eta^N_t$ corresponds to a branching event in $\eta^{N,k}_t$ (though not the other
  way around, as some bad particles in $\eta^{N,k}_t$ branch independently of $\eta^N_t$).

  At time $\frac{m+1}{2^k}$ we need to shave off the extra mass in $\eta^{N,k}_t$. Observe
  that, by our construction, we may erase particles of each of the three types. After erasing we
  relabel all remaining dangerous particles as bad by settting
  \[G^{N,k}_{(m+1)/2^k}=\eta^{N,k}_{(m+1)/2^k}\cap\,\eta^{N}_{(m+1)/2^k},\,
  B^{N,k}_{(m+1)/2^k}=\eta^{N,k}_{(m+1)/2^k}\setminus\eta^{N}_{(m+1)/2^k}
  \text{ and }D^N_{(m+1)/2^k}=\emptyset.\] In particular we see that the
  condition \eqref{eq:condCoupl} holds at time $\frac{m+1}{2^k}$, allowing us to continue
  our inductive coupling.

  We claim that the total number of bad particles after shaving and
  relabeling is bounded by the number of bad particles right before shaving:
  \begin{equation}
    |B^{N,k}_{(m+1)/2^k}|\leq|B^{N,k}_{(m+1)/2^k-}|.\label{eq:goodBadParts}
  \end{equation}
 To see where the inequality comes from observe first that
  \[B^{N,k}_{(m+1)/2^k}\subseteq B^{N,k}_{(m+1)/2^k-}\cup D^{N,k}_{(m+1)/2^k-}.\]
  Now each particle in $D^{N,k}_{(m+1)/2^k-}$ is associated with a branching time $T^N_i$ at
  which the number of particles in $\eta^{N,k}_{T^N_i}$ was increased by one. Therefore to each
  dangerous particle there corresponds some particle which will be erased when shaving;
  the corresponding particle to be erased is possibly the dangerous particle itself (in
  which case this particle will disappear after shaving so it will not be in
  $B^{N,k}_{(m+1)/2^k}$ after relabeling), and otherwise it has to be a bad particle
  because all good particles are to the right of any dangerous particle. In this way we
  continue the coupling until time 1.

  Fix a dyadic subinterval $[\frac{m}{2^k},\frac{m+1}{2^k})$. We claim that on this time
  interval the pair $(|D^{N,k}_t|,|B^{N,k}_t|)$ is stochastically dominated by a process
  $(d^k_t,b^k_t)$ which evolves as follows:
  \begin{center}
    \begin{tabular}{ccclc}
      $d^k_t$ & $\longrightarrow$ & $d^k_t+1$\hspace{.3in} & at rate & $N$ \\
      $b^k_t$ & $\longrightarrow$ & $b^k_t+1$\hspace{.3in} & at rate & $d^k_t+b^k_t$
    \end{tabular}
  \end{center}
  with initial conditions $d^k_{m/2^k}=0$ and $b^k_{m/2^k}=|B^{N,k}_{m/2^k}|$. In fact,
  bad particles increase by one when either a dangerous or a bad particle branches (so the
  second rate is actually the correct one), while dangerous particles are created as a
  consequence of some (but generally not all) of the branchings in $\eta^N_t$, which
  occur at rate $N$.
  An elementary calculation then shows that, for $h>0$,
  \begin{align}
    \ee(d^k_{t+h}-d^k_t)&=Nh+o(h)\\
    \ee(b^k_{t+h}-b^k_t)&=\ee(d^k_t+b^k_t)h+o(h).
  \end{align}
  Then $\ee(d^k_t)=N(t-\frac{m}{2^k})$ and thus dividing by $h$ and taking $h\to0$ we see
  that $\ee(b^k_t)$ must solve
  \begin{equation}\label{eq:diffbkt}
    \frac{d\ee(b^k_t)}{dt}=N\left(t-\frac{m}{2^k}\right)+\ee(b^k_t)
  \end{equation}
  for $t\in[\frac{m}{2^k},\frac{m+1}{2^k})$. The solution of this ODE satisfies 
  \begin{equation}\label{eq:bk}
    \ee(b^k_t)=\left(\ee(b^k_{m/2^k})+N\right)e^{t-\frac{m}{2^k}}-N(t-\tfrac{m}{2^k}+1)
      \leq \ee(b^k_{m/2^k})e^{2^{-k}}+2^{-2k}N,
    \end{equation}
  where we have used the inequality $e^{x}-1-x\leq x^2$ for $x\in[0,1]$ and the fact that
  $t-\frac{m}{2^k}\leq\frac{1}{2^k}$.  Since $b^k_0=0$ we deduce that
  $\ee(b^k_{(1/2^k)-})\leq N2^{-2k}$. At time $\frac{1}{2^k}$
  we need to shave off the extra mass in $\eta^{N,k}_{(1/2^k)-}$ and this leaves us with
  $d^k_{1/2^k}=0$ and $b^k_{{1/2^k}}\leq b^k_{(1/2^k)-}$ by
  \eqref{eq:goodBadParts}. Repeating this argument we get
    $\ee(b^k_{2/2^k})\leq\ee(b^k_{(2/2^k)-})\leq N[2^{-2k}e^{2^{-k}}+2^{-2k}]$,
  and inductively we deduce that
  \begin{equation} \label{eq:eeBt}
    \ee(b^k_{m/2^k})\leq N\left[\sum_{j=0}^{m-1}e^{j2^{-k}}\right]2^{-2k}
    =N\frac{1-e^{m2^{-k}}}{1-e^{2^{-k}}}2^{-2k}
    \leq N\frac{e-1}{e^{2^{-k}}-1}2^{-2k},
  \end{equation}
  where we used the fact that $m\leq 2^k-1$. Therefore, for
  $t\in[\frac{m}{2^k},\frac{m+1}{2^k})$ we have, using \eqref{eq:bk},
  \[\ee(b^k_t)\le N\frac{e-1}{e^{2^{-k}}-1}2^{-2k}e^{2^{-k}}+N2^{-2k}\]
  while, we recall, we also have $\ee(d^k_t)=N(t-\frac{m}{2^k})\leq N2^{-k}$.
  Since $|\eta^N_{t}\Delta\eta^{N,k}_{t}|\leq 2(d^k_t+b^k_t)$, the result follows.
\end{proof}
 
\begin{proof}[Proof of Proposition \ref{prop:cvNuNk:k}]
  Fix $k>0$ for a moment and assume that $t\in[\frac{m}{2^k},\frac{m+1}{2^k})$. Using the
  coupling introduced in Lemma \ref{lem:cvMassNuNk} we have
  \begin{align}
    \ee\!\left(\left|F^{N,k}(t,x)-F^N\!(t,x)\right|\right)
    &=\ee\!\left(\left|\frac{1}{N}\sum_{i=1}^N\left(\uno{\eta^{N,k}_t(i)\geq
          x}-\uno{\eta^N_t\!(i)\geq x}\right)\right|\right)\\
    &\leq\ee\!\left(\frac{1}{N}\sum_{i=1}^N\uno{\eta^{N,k}_t(i)\neq\eta^N_t\!(i)}\right)\\
    &\leq\frac{e-1}{e^{2^{-k}}-1}2^{-2k+1}e^{2^{-k}}+2^{-k+2},
  \end{align}
  so we deduce by Markov's inequality that
  \[\pp\!\left(\left|F^{N,k}(t,x)-F^N\!(t,x)\right|>\ep\right)
  \leq\frac{1}{\ep}\left[\frac{e-1}{e^{2^{-k}}-1}2^{-2k+1}e^{2^{-k}}+2^{-k+2}\right]
  \xrightarrow[k\to\infty]{}0\] uniformly in $N$.
\end{proof}

\begin{lem}\label{lem:FkIncr}
  \[F^k(t,x)\geq F^{k+1}(t,x)\quad\text{for all }t\in[0,1],\,x\in\rr,\,k\geq1.\]
\end{lem}

\begin{proof}
  Fix $k\geq1$ and $x\in\rr$. The result is trivial at $t=0$. We will work inductively on
  the intervals $(\frac{m}{2^{k}},\frac{m+1}{2^{k}}]$. Take $0\leq m\leq2^k-1$ and assume
  that
  \[F^k(\tfrac{m}{2^{k}},x)\geq F^{k+1}(\tfrac{m}{2^k},x).\] Then
  writing
  $H=F^k-F^{k+1}$ we have for
  $\frac{m}{2^{k}}<t<\frac{m}{2^{k}}+\frac{1}{2^{k+1}}=\frac{2m+1}{2^{k+1}}$ that
  \begin{equation}
    \begin{aligned}
      \frac{\partial H}{\partial
        t}(t,x)&=\int_x^\infty\!\int_{-\infty}^{\infty}\!\left[f^{k}(t,u)-f^{k+1}(t,u)\right]\!\rho(y-u)\,du\,dy\\
      &=\int_{-\infty}^\infty\!\int_x^\infty\!\left[f^k(t,y-v)-f^{k+1}(t,y-v)\right]\!\rho(v)\,dy\,dv\\
      &=\int_{-\infty}^\infty\!H(t,x-v)\rho(v)\,dv=\int_{-\infty}^\infty\!H(t,z)\rho(x-z)\,dz,
    \end{aligned}\label{eq:H}
  \end{equation}
  so $H$ satisfies \eqref{eq:FB1} on this interval and thus, since
  $H(\tfrac{m}{2^{k}},\cdot)\geq0$, we get $H(t,\cdot)\geq0$ for
  $t\in(\frac{m}{2^k},\frac{2m+1}{2^{k+1}})$. At time
  $\frac{2m+1}{2^{k+1}}$ the density $f^{k+1}$ is shaved
  off, leaving $F^k(\tfrac{2m+1}{2^{k+1}},x)\geq F^{k+1}(\tfrac{2m+1}{2^{k+1}},x)$.
  Repeating the above argument we get
  \begin{equation}
    F^k(t,x)\geq F^{k+1}(t,x)\quad\text{for all }t\in(\tfrac{m}{2^{k}},\tfrac{m+1}{2^{k}}).\label{eq:preFkDec}
  \end{equation}
  Now at time $\frac{m+1}{2^{k+1}}$ both densities $f^k$ and $f^{k+1}$ are shaved off, say
  at points $x_k$ and $x_{k+1}$, respectively. Then by \eqref{eq:preFkDec}, $x_k\geq
  x_{k+1}$, and thus \eqref{eq:preFkDec} holds at $t=\tfrac{m+1}{2^{k+1}}$ as well.
\end{proof}

Since $F^k(t,x)$ is decreasing and positive, we can define
\begin{equation}
F(t,x)=\lim_{k\to\infty}F^k(t,x).\label{eq:defF}
\end{equation}
It is obvious that for each given $t$, $F(t,\cdot)$
is non-increasing and its range is $[0,1]$.

\begin{prop}\label{prop:cvFN}
  For every $t\in[0,1]$ and $x\in\rr$, $F^N\!(t,x)$ converges almost surely and in $L^1$ as
  $N\to\infty$ to $F(t,x)$.
\end{prop}

\begin{proof}
  First observe that, for fixed $t\in[0,1]$ and $x\in\rr$ and since $F^N\!(t,x)\leq1$, the
  sequence of random variables $\big(F^N\!(t,x)\big)_{N>0}$ is uniformly integrable, so it
  is enough to show that $F^N\!(t,x)\to F(t,x)$ in probability.

  Fix $\ep>0$. Use Proposition \ref{prop:cvNuNk:N} to choose, for each $k>0$, an $N_k>0$
  so that
  \[\pp\!\left(\left|F^{N,k}(t,x)-F^k(t,x)\right|>\frac{\ep}{2}\right)<\frac{1}{k}\]
  for every $N\geq N_k$ and $N_k\uparrow\infty$. Define $k_N$ as follows: $k_N=1$ for
  $N<N_1$ and $k_N=k$ for $N_k\leq N\leq N_{k+1}$. We have
  \begin{multline}
    \pp\!\left(\left|F^{N,k_N}(t,x)-F(t,x)\right|>\ep\right)
    \leq\pp\!\left(\left|F^{N,k_N}(t,x)-F^{k_N}(t,x)\right|>\frac{\ep}{2}\right)\\
    +\pp\!\left(\left|F^{k_N}(t,x)-F(t,x)\right|>\frac{\ep}{2}\right).
  \end{multline}
  By the definition of $k_N$, the first term on the right hand side is less than $1/k_N$,
  while by \eqref{eq:defF} the second one is 0 when $k_N$ is large enough. We deduce that
  $F^{N,k_N}(t,x)$ converges in probability as $N\to\infty$ to $F(t,x)$.
  
  To finish the proof write
  \begin{multline}
    \pp\!\left(\left|F^N\!(t,x)-F(t,x)\right|>\ep\right)
    \leq\pp\!\left(\left|F^N\!(t,x)-F^{N,k_N}(t,x)\right|>\frac{\ep}{2}\right)\\
    +\pp\!\left(\left|F^{N,k_N}(t,x)-F(t,x)\right|>\frac{\ep}{2}\right).
  \end{multline}
  We already know that the second term on the right hand side goes to 0, while the first
  one goes to 0 thanks to Proposition \ref{prop:cvNuNk:k} (here we use the fact that the
  convergence is uniform in $N$).
\end{proof}

Recall the definition in \eqref{eq:defXmk} of the shaving points $X^k_m$ and let
$X^k\!:[0,1]\longrightarrow\rr$ be the corresponding linear interpolation, that is,
\[X^k(t)=X^k_m+\frac{X^k_{m+1}-X^k_m}{2^k}\left(t-\tfrac{m}{2^k}\right) \qquad\text{ for
}\tfrac{m}{2^k}<t\leq\tfrac{m+1}{2^k}.\]

\begin{lem}\label{lem:cvToGamma}
  $X^k(t)$ converges uniformly in $[0,1]$ to a continuous function $\gamma(t)$.
\end{lem}

\begin{proof}
  We will start by showing that the sequence of functions $\big(X^k\big)_{k>0}$ is
  relatively compact.  By the Arzel\`a-Ascoli Theorem, we only need to show that our
  sequence is uniformly bounded and equicontinuous.

  Observe that, for each given $k$, $X^k(t)$ is increasing. Indeed, it is enough to show
  that $X^k_m\leq X^k_{m+1}$ for $0\leq m<2^k$, and this follows from the fact that, if
  $f^k(\tfrac{m}{2^k},\cdot)\geq0$, then \eqref{eq:FB1} implies that $f^k(t,\cdot)\geq
  f^k(\tfrac{m}{2^k},\cdot)$ for $\tfrac{m}{2^k}<t<\tfrac{m+1}{2^k}$. Therefore
  \[\sup_{k>0}\sup_{t\in[0,1]}X^k(t)=\sup_{k>0}X^k(1).\]
  To show that this last supremum is finite, observe that $\hf(t,x)$ (which was defined in
  Proposition \ref{prop:nuhat}) satisfies $\hf(t,x)\geq f^k(t,x)$ for all $k$.  On the
  other hand we know by \eqref{eq:massHf} that $\int_{-\infty}^\infty\hf(1,x)\,dx=e$.
  Therefore if we let $M>0$ be such that $\int_M^\infty\hf(1,x)\,dx<1$ we deduce that
  $X^k(1)\leq M$ for all $k$ and the uniform boundedness follows.

  For the equicontinuity we need to show that given any $\ep>0$ there is a $\delta>0$ such
  that
  \[\sup_{k>0}\left|X^k(t)-X^k(s)\right|<\ep\]
  whenever $|t-s|<\delta$. Assume that $s<t$, fix $k$ for a moment and let $\frac{l}{2^k}$
  and $\frac{m}{2^k}$ be the dyadic numbers immediately to the right of $s$ and $t$,
  respectively (here we assume $k$ is large enough so that $m\vee l<2^k$). Then
  \begin{equation}
    \begin{aligned}
      \left|X^k(t)-X^k(s)\right|
      &\leq\left|X^k(t)-X^k(\tfrac{m}{2^k})\right|+\left|X^k(\tfrac{m}{2^k})-X^k(\tfrac{l}{2^k})\right|
      +\left|X^k(\tfrac{l}{2^k})-X^k(s)\right|\\
      &\leq\left|X^k(\tfrac{m+1}{2^k})-X^k(\tfrac{m}{2^k})\right|+\left|X^k(\tfrac{m}{2^k})-X^k(\tfrac{l}{2^k})\right|
      +\left|X^k(\tfrac{l+1}{2^k})-X^k(\tfrac{l}{2^k})\right|.
    \end{aligned}\label{eq:XtXs}
  \end{equation}
  Now for any $p,q\in\{0,\dotsc,2^k\}$ with $q\geq p$ we have
  \begin{multline}\label{eq:decFk}
    \int_{X^k(\frac{p}{2^k})+\ep}^\infty\!f^k(\tfrac{q}{2^k},y)\,dy
    =\int_{X^k(\frac{p}{2^k})}^\infty\!f^k(\tfrac{p}{2^k},y)\,dy
    +\int_{X^k(\frac{p}{2^k})}^\infty\!\left[f^k(\tfrac{q}{2^k},y)-f^k(\tfrac{p}{2^k},y)\right]\,dy\\
    -\int_{X^k(\frac{p}{2^k})}^{X^k(\frac{p}{2^k})+\ep}\!f^k(\tfrac{q}{2^k},y)\,dy.
  \end{multline}
  The first term on the right hand side equals 1. The second term corresponds to the
  amount of mass accumulated by $f^k$ to the right of $X^k(\frac{p}{2^k})$ on the time
  interval $(\frac{p}{2^k},\frac{q}{2^k}]$. Using \eqref{eq:FB1} it is not hard to see
  that this is bounded by the same quantity with $f^k$ replaced by $\hf$, so using
  \eqref{eq:massHf} we get the bound
  \begin{equation}
    \int_{X^k(\frac{p}{2^k})}^\infty\!\left[\hf(\tfrac{q}{2^k},y)-\hf(\tfrac{p}{2^k},y)\right]\,dy
    \leq e^\frac{q}{2^k}-e^\frac{p}{2^k}\leq e\frac{q-p}{2^k}.\label{eq:bdMassHat}
  \end{equation}
  On the other hand, using the fact that $X^k(t)$ is increasing, it is clear that
  $f^k(\frac{q}{2^k},x)\geq f^k(0,x)=f_0(x)$ for $x\geq X^k(\frac{q}{2^k})$. Therefore the
  last term on \eqref{eq:decFk} is greater than or equal to
  \[\int_{X^k(\frac{p}{2^k})}^{X^k(\frac{p}{2^k})+\ep}\!f_0(y)\,dy.\]
  Now $X^k(\frac{p}{2^k})$ is non-negative and bounded by $M$ by the preceding arguments,
  so the last integral is at least
  \[L=\inf_{x\in[0,M]}\int_x^{x+\ep}\!f_0(y)\,dy>0\] where we used the fact that $f_0$ is
  strictly positive on the positive half-line. Putting the last two bounds together with
  \eqref{eq:decFk} we get
  \begin{equation}
    \int_{X^k(\frac{p}{2^k})+\ep}^\infty\!f^k(\tfrac{q}{2^k},y)\,dy
    \leq 1+e\frac{q-p}{2^k}-L.\label{eq:massFkEp}
  \end{equation}
  Now $|t-s|<\delta$ implies that $\frac{m-l}{2^k}\leq\delta+\frac{1}{2^k}$, and thus we
  deduce that
  \[\int_{X^k(\frac{p}{2^k})+\ep}^\infty\!f^k(\tfrac{q}{2^k},y)\,dy<1\]
  for small enough $\delta$ and large enough $k$ and for
  $(p,q)\in\{(l,l+1),(m,m+1),(l,m)\}$. The preceding means that if $\delta$ is small
  enough and $K$ is large enough then $|X^k(\tfrac{q}{2^k})-X^k(\frac{p}{2^k})|<\ep$ for
  $k\geq K$ and for these three pairs $(p,q)$. Using \eqref{eq:XtXs} we obtain
  \[\sup_{k\geq K}\left|X^k(t)-X^k(s)\right|<\ep\]
  if $|t-s|<\delta$ and $\delta$ is small enough. Since the functions $X^k$ are all
  uniformly continuous (on $[0,1]$), it is clear that, by choosing $\delta$ even smaller
  if necessary, the same will hold also for $k=1,\dotsc,K-1$. This finishes the proof of
  the equicontinuity.
  
  The last thing we need to show is that our sequence has a unique limit point. Consider
  two convergent subsequences $X^{n_k}\to\gamma_1$ and $X^{m_k}\to\gamma_2$. Let
  $t=\frac{i}{2^l}$ be any dyadic number in $[0,1]$ and assume that $k$ is large enough so
  that $n_k\wedge m_k\geq l$. Recall from the proof of Lemma \ref{lem:FkIncr} that
  $X^k(t)$ is non-increasing in $k$ for each fixed $t\in[0,1]$. Since $F^{n_k}(t,x)=1$ for
  all $x\leq X^{n_k}(t)$ we deduce that
  \begin{equation}\label{eq:bdFnk}
    F^{n_k}(t,x)=1\qquad\text{for all }x\leq\gamma_1(t).
  \end{equation}
  Now given any $k$ there is a $k'$ such that $n_{k'}\geq m_k$, so by Lemma
  \ref{lem:FkIncr} we get
  \[1=F^{n_{k'}}(t,x)\leq F^{m_k}(t,x)\leq1\qquad\text{for all }x\leq\gamma_1(t).\] This
  means that $X^{m_k}(t)\geq\gamma_1(t)$ for all large enough $k$, and taking $k\to\infty$
  we deduce that $\gamma_2(t)\geq\gamma_1(t)$. By symmetry we get
  $\gamma_1(t)\geq\gamma_2(t)$. This gives $\gamma_1(t)=\gamma_2(t)$ for all dyadic
  $t\in[0,1]$, and now the uniqueness follows from the continuity of $\gamma_1$ and
  $\gamma_2$.
\end{proof}

\subsection{Properties of \texorpdfstring{$F$}{F} and proof of the theorem}

To finish the proof of the Theorem \ref{thm:fb} we need to show that $F$ has a density
which satisfies \FB and the rest of the requirements of the theorem and then extend the
convergence to the measure-valued process $\nu^N_t$. The first step in doing that will be
to derive an equation satisfied by $F$.

Suppose that $(g(t,x),\gamma(t))$ solves \FB and let $G(t,x)=\int_x^\infty
g(t,y)\,dy$. Then it is not difficult to check, repeating the arguments leading to
\eqref{eq:H}, that $(G(t,x),\gamma(t))$ must solve the following free boundary problem
\hypertarget{FBp}{{\rm(FB${}^\prime$)}}:
\begin{gather}
  \frac{\partial G}{\partial t}(t,x)=\int_{-\infty}^\infty\!G(t,y)\rho(x-y)\,dy
  \quad\forall\,x>\gamma(t)\tag{FB1${}^\prime$}\label{eq:FBp1}\\
  G(t,x)=1\quad\forall\,x\leq\gamma(t)\tag{FB2${}^\prime$}\label{eq:FBp2}
\end{gather}
with initial condition $G(0,x)=\int_x^\infty f_0(y)\,dy$. Moreover, if
$(G(t,x),\gamma(t))$ solves \FBp and $G(t,\cdot)$ is absolutely continuous for all $t$,
then $(g(t,x),\gamma(t))$, where $g(t,\cdot)$ is the density of $G(t,\cdot)$, must solve
\FB.

\begin{prop}\label{prop:FBp}
  $F(t,x)$ is differentiable in $t$ for all $x>\gamma(t)$ and it satisfies \FBp.
\end{prop}

\begin{proof}
  We already proved (see \eqref{eq:bdFnk}) that $(F(t,x),\gamma(t))$ satisfies
  \eqref{eq:FBp2}. For $x>\gamma(t)$ and by the definition of $F^k(t,x)$ (which implies
  that $F^k(t,x)$ is differentiable inside each dyadic subinterval) we may write
  \begin{align}
    F^k(t,x)&=F^k(0,x)+\sum_{m=1}^{n_k(t)}\left[F^k(\tfrac{m}{2^k},x)-F^k(\tfrac{m-1}{2^k},x)\right]
    +\left[F^k(t,x)-F^k(n_k(t),x)\right],\\
    &=F^k(0,x)+\sum_{m=1}^{n_k(t)}\left[\int_{(\frac{m-1}{2^k},\frac{m}{2^k})}\!\frac{\partial
        F^k}{\partial s}(s,x)\,ds\right]+\int_{(\frac{n_k(t)}{2^k},t)}\!\frac{\partial
      F^k}{\partial
      s}(s,x)\,ds\\
    &\hspace{0.35in}+\sum_{m=1}^{n_k(t)}\left[F^k(\tfrac{m}{2^k},x)-F^k((\tfrac{m}{2^k})-,x)\right]
    +\left[F^k(t,x)-F^k(t-,x)\right],
  \end{align}
  where $n_k(t)=\frac{\lfloor2^kt\rfloor}{2^k}$. Recalling that
  $X^k(t)\downarrow\gamma(t)$, we can take $k$ large enough so that $\gamma(t)\leq
  X^k(t)<x$. Since $X^k(s)$ is increasing in $s$ we deduce that $X^k(\frac{m}{2^k})<x$ for
  $m=1,\dotsc,n_k(t)$, and therefore all the terms in the last line above are 0. On the
  other hand, observe that $F^k$ must solve \eqref{eq:FBp1} on each dyadic subinterval,
  which can be checked repeating again the calculations in \eqref{eq:H}. Therefore,
  \begin{equation}
    \begin{aligned}
      F^k(t,x)&=F^k(0,x)+\sum_{m=1}^{n_k(t)}\int_{(\frac{m-1}{2^k},\frac{m}{2^k})}\!
      \int_{-\infty}^\infty\!F^k(s,y)\rho(x-y)\,dy\,ds\\
      &\hspace{2.2in}+\int_{(\frac{n_k(t)}{2^k},t)}\!\int_{-\infty}^\infty\!F^k(s,y)\rho(x-y)\,dy\,ds\\
      &=F^k(0,x)+\int_0^t\!\int_{-\infty}^\infty\!F(s,y)\rho(x-y)\,dy\,ds\\
      &\hspace{1in}+\sum_{m=1}^{n_k(t)}\int_{(\frac{m-1}{2^k},\frac{m}{2^k})}\!
      \int_{-\infty}^\infty\!\left[F^k(s,y)-F(s,y)\right]\!\rho(x-y)\,dy\,ds\\
      &\hspace{1in}+\int_{(\frac{n_k(t)}{2^k},t)}\!\int_{-\infty}^\infty\!\left[F^k(s,y)-F(s,y)\right]\!\rho(x-y)\,dy\,ds.
    \end{aligned}\label{eq:telesFk}
  \end{equation}
  Now for fixed $y$, $F^k(\cdot,y)$ is a decreasing sequence converging to $F(\cdot,y)$,
  so Dini's Theorem implies that
  \[\Delta_k(y)=\sup_{t\in[0,1]}\left|F^k(t,y)-F(t,y)\right|\xrightarrow[k\to\infty]{}0.\]
  The sum of the terms on the last two lines of \eqref{eq:telesFk} is bounded by
  \begin{align}
    \sum_{m=1}^{n_k(t)}\int_{(\frac{m-1}{2^k},\frac{m}{2^k})}\!&\int_{-\infty}^\infty\!\Delta_k(y)\rho(x-y)\,dy\,ds
    +\int_{(\frac{n_k(t)}{2^k},t)}\!\int_{-\infty}^\infty\!\Delta_k(y)\rho(x-y)\,dy\,ds\\
    &\leq t\int_{-\infty}^\infty\!\Delta_k(y)\rho(x-y)\,dy,
  \end{align}
  and this last integral goes to 0 as $k\to\infty$ by the dominated convergence theorem,
  because using \eqref{eq:massHf} we get
  $\Delta_k(y)\leq\sup_{t\in[0,1]}F^k(t,y)\leq\int_y^\infty\!\hf(1,z)\,dz\leq e$.  Using
  this and taking $k\to\infty$ in \eqref{eq:telesFk} we get
  \begin{equation}
    F(t,x)=F(0,x)+\int_0^t\!\int_{-\infty}^\infty\!F(s,y)\rho(x-y)\,dy\,ds.\label{eq:intFormF}
  \end{equation}
  
  To finish the proof it is enough to show that
  the mapping $s\mapsto \int_{-\infty}^\infty F(s,y)\rho(x-y)\,dy$ is continuous, since if
  that is the case then we can differentiate \eqref{eq:intFormF} and deduce
  \eqref{eq:FBp1}. This actually follows easily from \eqref{eq:intFormF}:
  \begin{equation}
    \begin{aligned}
      \Bigg|\int_{-\infty}^\infty\!F(s+h,y)\rho(x-y)&\,dy-\int_{-\infty}^\infty\!F(s,y)\rho(x-y)\,dy\Bigg|\\
      &=\int_{-\infty}^\infty\!\int_s^{s+h}\!\int_{-\infty}^\infty\!F(r,z)\rho(y-z)\rho(x-y)\,dz\,dr\,dy\\
      &\leq\int_s^{s+h}\!\int_{-\infty}^\infty\!\int_{-\infty}^\infty\!\rho(y-z)\rho(x-y)\,dz\,dy\,dr=h.\qedhere
    \end{aligned}\label{eq:diffF}
  \end{equation}
\end{proof}

Let $\nu_t$ be the probability measure defined by $\nu_t([x,\infty))=F(t,x)$. Since $F$
satisfies \FBp we have that for every $b>a>\gamma(t)$,
\[\frac{d}{dt}\nu_t([a,b])=\int_{-\infty}^\infty\nu_t([a-y,b-y])\rho(y)\,dy\]
and thus by standard measure theory arguments we deduce that for every bounded and
measurable $\varphi$ with support contained in $(\gamma(t),\infty)$,
\begin{equation}
\frac{d}{dt}\int_{-\infty}^\infty\!\varphi(y)\nu_t(dy)
=\int_{-\infty}^\infty\!\int_{-\infty}^\infty\!\varphi(x+y)\rho(y)\,\nu_t(dx)\,dy\label{eq:nuPhi}
\end{equation}
(cf. \eqref{eq:eqNuHat}). Now if $A\subseteq(\gamma(t),\infty)$ has zero Lebesgue measure and the
support of $\varphi$ is contained in $A$, then the right-hand side above is 0 and we
deduce that $\nu_t(A)$ is constant. Since $\nu_0(A)=\int_A f_0(x)\,dx=0$, we have proved
that $\nu_t$ is absolutely continuous with respect to the Lebesgue measure. We will
denote its density by $f(t,\cdot)$, and we obviously have $F(t,x)=\int_x^\infty f(t,y)\,dy$.

At this point we are ready to finish the proof of Theorem \ref{thm:fb} by showing that $F$
satisfies the desired properties and then using the convergence of the tail distributions
$F^N(t,x)$ to obtain the convergence in distribution of the process $\nu^N_t$ in
$D([0,1],\cp)$.

\begin{proof}[Proof of Theorem \ref{thm:fb}]
  By Proposition \ref{prop:cvFN} we know that $F^N\!(t,x)\to F(t,x)$ almost surely as
  $N\to\infty$ and that $F$ can be written in terms of the integral of $f$. Now, by
  \eqref{eq:intFormF},
  \begin{equation}
    \int_x^\infty\!f(t,y)\,dy=\int_x^\infty\!f_0(y)\,dy+\int_0^t\!\int_{-\infty}^\infty\!\int_y^\infty\!f(s,z)\rho(x-y)\,dz\,dy\,ds
  \end{equation}
  for $x>\gamma(t)$, so for any $h>0$ we have
  \begin{multline}\label{eq:intDens}
    \frac{1}{h}\left[\int_{x+h}^\infty\!f(t,y)\,dy-\int_x^\infty\!f(t,y)\,dy\right]\\
    =-\frac{1}{h}\int_x^{x+h}\!f_0(y)\,dy
    -\int_0^t\!\int_{-\infty}^\infty\!\frac{1}{h}\int_{x-u}^{x+h-u}\!f(s,z)\,dz\,\rho(u)\,du\,ds.
  \end{multline}
  The first term on the right-hand side goes to $f_0(x)$ as $h\to0$ (recall that
  $f_0(x)$ is continuous for $x>0$). The second term
  goes to $\int_0^t\!\int_{-\infty}^\infty f(s,x-u)\rho(u)\,du\,ds$ by the dominated
  convergence theorem. On the other hand, the left-hand side of \eqref{eq:intDens} goes to
  $\tfrac{\partial F}{\partial x}(t,x)$, which equals $-f(t,x)$ for almost every
  $x>\gamma(t)$. We deduce that
  \begin{equation}
    f(t,x)=f_0(x)+\int_0^t\!\int_{-\infty}^\infty\!f(s,y)\rho(x-y)\,dy\,ds\label{eq:formIntDens}
  \end{equation}
  for almost every $x>\gamma(t)$.
  The right hand side is continuous in $x$, and hence so is the left hand side, which implies that \eqref{eq:formIntDens} holds for every $x>\gamma(t)$. 
  Moreover, since the above convergence can be achieved uniformly for $t$ in compact
  intervals, it is easy to see that, if $f(t,x)$ is continuous in $t$ for $x\neq\gamma(t)$
  (as we will show next), then it is actually jointly continuous in $t$ and $x$ outside
  the curve $\{(t,\gamma(t))\!:t\geq0\}$.  The fact that $f(t,x)$ is differentiable (and
  thus continuous) in $t$ for $x\neq\gamma(t)$ follows easily from \eqref{eq:formIntDens}
  by repeating the arguments in \eqref{eq:diffF}. $\gamma$ is strictly increasing because,
  according to the evolution defined by \FB, $f(t,x)$ always increases when $x>\gamma(t)$.
  We also have that $f(t,x)>0$ for $x>\gamma(t)$ thanks to the facts that
  $f_0(x)>0$ for $x>0$ and that $\gamma$ is increasing.

  The conclusion from all this is that $(f(t,x),\gamma(t))$ satisfies \FB.
  Such a solution is unique thanks to Lemma 5.2 in \cite{atar}.

  To finish the proof we need to show that the sequence of processes $\nu^N_t$ converges in distribution in   $D([0,1],\cp)$ to the deterministic process $\nu_t$ in this space defined by having its densities evolve according to \FB.
  To this end it is enough to prove that this sequence is tight. In
  fact, if $\nu^{N_k}_t$ is any convergent subsequence, then its limit $\nu_t$ is
  completely defined by its tail distribution at each time $t$, which we know must be
  $F(t,\cdot)$. To show that $\nu^N_t$ is tight it is enough, by Theorem 2.1 of
  \citet{roel}, to show that for any continuous and bounded function $\varphi$ on $\rr$
  the sequence of real-valued processes
  $\langle\nu^N_t,\varphi\rangle=\int_{-\infty}^\infty\varphi(x)\,\nu^N_t(dx)$ is tight in
  $D([0,T],\rr)$.  Fix one such function $\varphi$. By Aldous' criterion (which we take
  from Theorem 2.2.2 in \citet{joffeMetiv} and the corollary that preceeds it in page 34),
  we need to prove that the following two conditions hold:
  \begin{enumerate}[label=(\roman*)]
  \item For every rational $t\in[0,T]$ and every $\ep>0$, there is an $L>0$ such that
    \[\sup_{N>0}\pp\!\left(|\langle\nu^N_t,\varphi\rangle|>L\right)\leq\ep.\]
  \item If $\mathfrak T^N_T$ is the collection of stopping times with respect to the
    natural filtration associated to $\langle\nu^N_t,\varphi\rangle$ that are almost
    surely bounded by $T$, then for every $\ep>0$
    \[\lim_{r\rightarrow0}\limsup_{N\rightarrow\infty}
    \sup_{\substack{s<r\\\tau\in\mathfrak
        T^N_T}}\pp\!\left(\left|\langle\nu^N_{(\tau+s)\wedge
          T},\varphi\rangle-\langle\nu^N_\tau,\varphi\rangle)\right|>\ep\right)=0.\]
  \end{enumerate}
  The first condition holds trivially in our case by taking $L>\|\varphi\|_\infty$. To get
  the second one fix $N>0$, $\ep>0$, $0<s<r$ and $\tau\in\mathfrak T^N_T$ and let $K$ be
  the number of branchings in $\eta^N_t$ on the interval $[\tau,(\tau+s)\wedge
  T]$. Observing that
  \[\left|\langle\nu^N_{(\tau+s)\wedge
      T},\varphi\rangle-\langle\nu^N_\tau,\varphi\rangle\right|\leq\frac{2\|\varphi\|_\infty}{N}K\]
  and $\ee(K)\leq Ns<Nr$, we deduce by Markov's inequality that
  \[\pp\!\left(\left|\langle\nu^N_{(\tau+s)\wedge
        T},\varphi\rangle-\langle\nu^N_\tau,\varphi\rangle)\right|>\ep\right)
  \leq\pp\!\left(\frac{2\|\varphi\|_\infty}{N}K>\ep\right)\leq\frac{2\|\varphi\|_\infty\ee(K)}{\ep
    N}<\frac{2\|\varphi\|_\infty r}{\ep}\] and (ii) follows.
\end{proof}

\begin{rem}\label{rem:flaweduniq} \emph{(Added April 2023)} ~ 
In the published version of this article an argument for the uniqueness of solutions of \FB is provided as part of the proof of Theorem \ref{thm:fb}.
The idea of the argument was to consider a given solution $(h(t,x),\sigma(t))$ and compare $h(t,x)$ and $f^k(t,x)$ by adapting the coupling introduced in the proof of Lemma \ref{lem:cvMassNuNk}.
However, it was pointed out to us recently by Rami Atar that the argument is flawed.
We explain the issue next.

\noindent
In the argument we considered three non-negative functions $g^k(t,x)$, $d^k(t,x)$ and $b^k(t,x)$, which evolve inside the first dyadic subinterval $(0,\frac{1}{2^k})$ according to the following system of equations:
  \begin{equation}
    \label{eq:sysK}
    \begin{aligned}
      \frac{\partial g^k}{\partial
        t}(t,x)&=\uno{x>\sigma(t)}\int_{-\infty}^{\infty}\!g^k(t,y)\rho(x-y)\,dy,\\
      g^k(t,x)&=0\qquad\text{for }x\leq\sigma(t),\\
      \frac{\partial d^k}{\partial
        t}(t,x)&=\uno{x\leq\sigma(t)}\int_{-\infty}^{\infty}\!g^k(t,y)\rho(x-y)\,dy,\\
      \frac{\partial b^k}{\partial
        t}(t,x)&=\int_{-\infty}^{\infty}\!\left[d^k(t,y)+b^k(t,y)\right]\rho(x-y)\,dy,
    \end{aligned}
  \end{equation}
with initial condition $g^k(0,x)=f_0(x)$ and $d^k(0,x)=b^k(0,x)=0$ for all $x$.
The rest of the construction was inductive, with the same system governing the evolution of the three functions inside each dyadic subinterval of $[0,1]$ together with a procedure for shaving off and rebalancing mass at each dyadic time.
The problem with the argument arises already with the construction of the system \eqref{eq:sysK} in $(0,\frac{1}{2^k})$, which is not consistent.
The idea was that the function $g^k(t,x)+d^k(t,x)+b^k(t,x)$ solves \eqref{eq:fK} there, as follows from adding the first and last two equations in \eqref{eq:sysK}, so that one has $f^k(t,x)=g^k(t,x)+d^k(t,x)+b^k(t,x)$.
However, \eqref{eq:sysK} as written is clearly contradictory: $\sigma$ is continuous and increasing with $\sigma(0)=0$, and thus for $x\in(0,\sigma(T)]$, $T\in(0,\frac{1}{2^k})$, we may integrate the first equation in $t$ on the interval $[0,T]$ to deduce that $g^k(T,x)\geq g^k(0,x)=f_0(x)>0$, contradicting the second equation.
On the other hand, if one modifies this construction so that the first equation defines $g^k(t,x)$ only for $x>\sigma(t)$ then this contradiction is removed but now the equality $f^k(t,x)=g^k(t,x)+d^k(t,x)+b^k(t,x)$ does not hold.

\noindent
Although one can expect there to be a way to fix the proof, it appears to us as if any such fix will require a comparison argument similar to the one contained in the proof of Lemma 5.2 in \cite{atar}, which already proves uniqueness (in a more general setting) without the need to resort to the above construction.
\end{rem}

\section{Proof of the results for the finite system}
\label{sec:proofs:finite}

Now we turn to the properties of the finite system. Recall the explicit construction of
$\eta^N_t$ we gave before the proof of Lemma \ref{lem:cvMassNuNk}: given an i.i.d. family
$(U^N_i)_{i\geq1}$ with uniform distribution on $\{1,\dotsc,N\}$, an i.i.d. family
$(R_i)_{i\geq1}$ with distribution $\rho$ and the jump times $(T^N_i)_{i\geq1}$ of a
Poisson process with rate $N$, we construct $\eta^N_t$ by letting $\eta^N_{T^N_i}(U^N_i)$
branch at time $T^N_i$, using $R_i$ for the displacement, erasing the leftmost particle, and
then relabeling the particles to keep the ordering.

This construction allows us to give a monotone coupling for two copies of the process. As
in the proof of Proposition \ref{prop:cvNuNk:N}, for $\mu,\nu\in\cm$ we will say that
$\mu\preceq\nu$ whenever $\mu([x,\infty))\leq\nu([x,\infty))$ for all $x\in\rr$. Observe
that if $\mu=\sum_{i=1}^{N_1}\delta_{x_i}$ with $x_1\geq\dotsm\geq x_{N_1}$ and
$\nu=\sum_{i=1}^{N_2}\delta_{y_i}$ with $y_1\geq\dotsm\geq y_{N_2}$ then $\mu\preceq\nu$
if and only if $N_1\leq N_2$ and $x_i\leq y_i$ for $i=1,\dotsc,N_1$. It is easy to check
that if $\eta^{N_1}_t$ and $\xi^{N_2}_t$ are two copies of our process (note that we allow
them to have different total number of particles) with $\eta^{N_1}_0\preceq\xi^{N_2}_0$
(i.e., in the sense that
$\sum_{i=1}^{N_1}\delta_{\eta^{N_1}_0(i)}\preceq\sum_{i=1}^{N_2}\delta_{\xi^{N_2}_0(i)}$),
then if we use the same branching times and displacements and the same uniform variables
for the particles in $\eta^{N_1}_t$ and the leftmost $N_1$ particles in $\eta^{N_2}_t$,
then we have $\eta^{N_1}_t\preceq\xi^{N_2}_t$ for all $t\geq0$.

For most of the proof of Theorem \ref{thm:finite} it will more convenient to work with the
discrete time version of our process $\eta^N_n$ which is defined as follows: at each time
step, choose one particle uniformly at random, and then branch that particle and remove
the leftmost particle among the $N+1$. The variables $U^N_i$ and $R_i$ can be used to
decide which particle to branch and where to send its offspring at each time step.

\begin{proof}[Proof of Theorem \ref{thm:finite}(a)]
  We will first prove that each of the two limits exists with probability 1 and in $L^1$
  and that the limits are non-random. We will do this for the discrete time process and
  leave to the reader the (easy) extension to the continuous time case. We borrow the
  proof from that of Proposition 2 in \citet{berGou}. Since it is simple we include it for
  convenience. We observe that, as in the cited proof, by translation invariance and the
  monotonicity property of the coupling discussed above, it is enough to prove the result when all
  particles start at the origin.

  The result is a consequence of the subadditive ergodic theorem. To see why, suppose we
  run the process up to time $k$, restart it with all $N$ particles at $\max\eta^N_k$, and
  then run it for an extra $l$ units of time. Then the resulting configuration will
  dominate the configuration that we would get by running the process for
  time $k+l$. To apply the subadditive ergodic theorem we will need to make this precise
  by introducing an appropriate coupling.

  Consider the variables $U^N_i$ and $R_i$ used to construct $\eta^N_n$. For
  each $k\geq0$ let $\big(\eta^N_{k,n}\big)_{n\geq0}$ be a copy of our process, started at
  $\eta_0$, constructed as follows: if $\eta^N_{k,n}$ is given then we let
  $\eta^N_{k,n+1}$ be specified by adding a particle at $\eta^N_{k,n}(U^N_{n+k})+R_{n+k}$
  and then removing the leftmost particle. That is, the index $n$ in $\eta^N_{k,n}$
  corresponds to time while the index $k$ indicates that the $k$-th copy of the process
  $\big(\eta^N_{k,n}\big)_{n\geq0}$ uses the random the variables $U^N_i$ and $R_i$
  starting from the $k$-th one. With this definition and in view of the preceding paragraph it is not difficult to see
  that for any $k,l\geq0$,
  \[\max\eta^N_{0,k+l}\leq\max\eta^N_{0,k}+\max\eta^N_{k,l}.\] Moreover, for any given
  $d\geq1$ the family $\big(\eta^N_{dm,d}\big)_{m\geq1}$ is i.i.d., because to compute
  $\eta^N_{dm,d}$ we only need to use the variables $U^N_i$ and $R_i$ for
  $i=dm,\dotsc,dm+d-1$. Also observe that the distribution of
  $\big(\eta^N_{k,n}\big)_{n\geq0}$ does not depend on $k$. Now for $0\leq k\leq n$ define
  $\xi^N_{k,n}=\eta^N_{k,n-k}$. Then using the above facts we see that
  \[\max\xi^N_{0,k+l}\leq\max\xi^N_{0,k}+\max\xi^N_{k,k+l},\] the family
  $\big(\!\max\xi^N_{dm,d(m+1)}\big)_{m\geq1}$ is i.i.d for any $d\geq1$ and the
  distribution of the sequence $\big(\!\max\xi^N_{k,n+k}\big)_{n\geq0}$ does not depend on
  $k$. It is not hard to check that $\max\xi^N_{k,n}$ satisfies the rest of the hypotheses
  of the subadditive ergodic theorem (see Theorem 6.6.1 in \citet{durrPTE}) and thus
  $\lim_{n\to\infty}\max\xi^N_{0,n}/n$ exists almost surely and in $L^1$, and moreover the
  limit is non-random. Since $\big(\xi^N_{0,n}\big)_{n\geq0}$ has the same distribution as
  $\big(\eta^N_n\big)_{n\geq0}$, the same holds for $\lim_{n\to\infty}\max\eta^N_n/n$.

  The above proof can be straightforwardly adapted to obtain the existence of the limit
  for $\min\eta^N_t/t$. To show that the two limits are equal it is enough to prove that
  $(\max\eta^N_t-\min\eta^N_t)/t\to0$ in probability as $t\to\infty$. Observe that if we
  follow the genealogy of the particle at $\max\eta^N_t$ back in time and go back in time
  $N$ generations then we will necessarily reach a particle that is not in $\eta^N_t$ (at
  time $t$), and that is thus to the left of $\min\eta^N_t$. If we call $X_t$ the position
  of this particle and let $N_t$ be the number of branchings in the system up to time $t$,
  then clearly $\max\eta^N_t-X_t\leq\sum_{i=N_t-N}^{N_t}|R_i|$. Thus for any $\ep>0$ we
  have that
  \begin{align}
    \pp\!\left(\left|\frac{\max\eta^N_t-\min\eta^N_t}{t}\right|>\ep\right)
    &\leq\pp\!\left(|R_i|>\sqrt{t}\text{ for some }N_t-N\leq i\leq N_t\right)
    +\uno{\frac{N\sqrt{t}}{t}>\ep}\\
    &=1-\pp\!\left(|R_1|\leq\sqrt{t}\right)^N+\uno{\frac{N}{\sqrt{t}}>\ep}
    \xrightarrow[t\to\infty]{}0.
  \end{align}

  The monotone coupling introduced above allows to deduce that $a_N$ is non-decreasing. On
  the other hand, in the case $N=1$ we have that $\eta^1_t(1)$ is simply a random walk
  jumping at rate 1 whose jump distribution is that of $R_1\vee0$. Therefore
  $\ee(\eta^1_t(1))=bt$ with $b=\int_0^\infty x\rho(x)\,dx>0$, and thus $a_N\geq a_1=b>0$
  for all $N\geq1$.
\end{proof}

To prove parts (b) and (c) of Theorem \ref{thm:finite} we will work with the
discrete time version of the shifted process:
$\Delta^N_n=(\Delta^N_n(1),\dotsc,\Delta^N_n(N))$ with
\[\Delta^N_n(j)=\eta^N_n(j)-\eta^N_n(N).\]

\begin{prop}\label{prop:harris}
  $\Delta^N_n$ is a positive recurrent Harris chain.
\end{prop}

\begin{proof}
  Following \citet{athNey}, in order to show that $\Delta^N_n$ is Harris recurrent we need
  to show that there is a set $A\subseteq\cx_N$ such that
  \begin{enumerate}[label=(\roman*)]
  \item $\pp^\xi(\tau_A<\infty)=1$ for all $\xi\in\cx_N$, where
    $\tau_A=\inf\{n\geq0\!:\Delta^N_n\in A\}$.
  \item There exists a probability measure $q$ on $A$, a $\lambda>0$ and a $k\in\nn$ so
    that $\pp^\xi(\Delta^N_k\in B)\geq\lambda q(B)$ for all $\xi\in A$ and all $B\subseteq
    A$.
  \end{enumerate}
  To achieve this, choose some $L>0$ so that $\delta=\rho((0,L))>0$ and let
  \[A=\left\{\xi\in\cx_N\!:\xi(i)-\xi(i+1)\in(0,L)\,\text{for
    }i=1,\dotsc,N-1\,\,\text{and}\,\,\xi(N)=0\right\}.\] Then for any initial condition
  $\Delta^N_0\in\cx_N$ we can get to $A$ in $N-1$ steps via the following path:
  at time 1 we choose to branch the rightmost particle (the one at $\Delta^N_0\!(1)$, which
  happens with probability $N^{-1}$) and send the newborn particle to a location
  $x_1\in(\Delta^N_0\!(1),\Delta^N_0\!(1)+L)$ (which happens with probability at least
  $\delta$). Next we branch the particle at $x_1$ and send the newborn particle to a
  location $x_2\in(x_1,x_1+L)$ (which happens with probability at least $\delta/N$). If we
  continue this for $N-1$ steps we will end up with a configuration in $A$, and thus
  \begin{equation}
    \pp^\xi\!\left(\Delta^N_{N-1}\in A\right)\geq\left(\frac{\delta}{N}\right)^{N-1}.\label{eq:hitA}
  \end{equation}
  The bound is independent of the initial condition $\xi$, so by the Borel-Cantelli Lemma
  it follows that (i) holds. Moreover, if $B\subseteq A$ is of the form
  $B=\{\xi\in\cx_N\!:\,\xi(i)-\xi(i+1)\in B_i\subseteq(0,L)\,\text{for
  }i=1,\dotsc,N-1\,\,\text{and}\,\,\xi(N)=0\}$, then the preceding argument implies that
  \[\pp^\xi\!\left(\Delta^N_{N-1}\in
    B\right)\geq\frac{\rho(B_1)\dotsm\rho(B_{N-1})}{N^{N-1}},\] so by taking
  $\lambda=N^{-N+1}$ and $q$ to be the normalized Lebesgue measure on the first $N-1$
  coordinates of the configurations in $A$, we deduce that (ii) also holds.

  To check that $\Delta^N_n$ is positive recurrent it is enough to check that
  $\sup_{\xi\in\cx_N}\ee^\xi(\tau_A)<\infty$. This follows from \eqref{eq:hitA}
  and the strong Markov property by writing
  \begin{align}
    \ee^\xi\!\left(\tau_A\right)&=\sum_{n\geq1}\pp^\xi(\tau_A\geq n)
    \leq N-1+\sum_{i\geq1}\sum_{n=iN}^{(i+1)N-1}\pp^\xi(\tau_A\geq n)\\
    &\leq N-1+(N-1)\sum_{i\geq1}\left[1-\left(\frac{\delta}{N}\right)^{N-1}\right]^i
    =(N-1)\left(\frac{N}{\delta}\right)^{N-1}<\infty
  \end{align}
  for any $\xi\in\cx_N$.
\end{proof}

 \begin{proof}[Proof of Theorem \ref{thm:finite}(b)]
   The result now follows from Proposition \ref{prop:harris}. The fact that $\Delta^N_n$
   is positive recurrent implies that the invariant measure whose existence is assured by
   the Harris recurrence is finite. The absolute continuity of $\mu_N$ is a direct
   consequence of Theorem \ref{thm:finite}(c), which we prove below, together with the
   fact that if the initial condition for $\Delta^N_t$ is absolutely continuous, then so
   is the distribution of the process at all times.
 \end{proof}
 
 \begin{proof}[Proof of Theorem \ref{thm:finite}(c)]
   Let $A\subseteq\cx_N$ and $k=N-1$ be the objects which we found satisfy (i) and (ii) in
   the proof of Proposition \ref{prop:harris}. It is enough to prove the result along the
   $k$ subsequences of the form $\big(\Delta^N_{km+j}\big)_{m\geq0}$ with $0\leq
   j<k$. Moreover, using the Markov property at time $j$ we see that it is enough to prove
   the result along the subsequence $\big(\Delta^N_{km}\big)_{m\geq0}$, which is an
   aperiodic recurrent Harris chain. The result for this subsequence follows from Theorem
   4.1(ii) in \citet{athNey} as long as we have that $\sup_\xi\pp^\xi(\tau_A>t)<1$ for
   some $t>0$, where $\tau_A$ is the hitting time of $A$. This follows easily from the
   estimate in \eqref{eq:hitA} (which is uniform in $\xi$).
 \end{proof}

 \section{Proof of Theorem \ref{thm:wave}}
 \label{sec:proof:wave}

 Recall that in this part we are assuming that $\rho$ (and hence its tail distribution
 $R$) has exponential decay and, consequently, that the moment generating function of
 $\rho$, $\phi(\theta)=\int_{-\infty}^\infty e^{\theta x}\rho(x)\,dx$, is finite for
 $\theta\in(-\Theta,\Theta)$ (see \eqref{eq:expDecayRho} and \eqref{eq:expDecayR}). Before
 getting started with the proof of Theorem \ref{thm:wave} we need to prove the claim we
 implicitly made in \eqref{eq:claimGoodSpeed}.

 \begin{lem}\label{lem:goodSpeed}
   \[\min_{\lambda\in(0,\Theta)}\frac{\phi(\lambda)}{\lambda}=a,\]
   where $a$ is the asymptotic speed defined in Theorem \ref{thm:limAN}. Moreover, letting
   $\ls\in(0,\Theta)$ be the number such that $\phi(\ls)/\ls=a$, we have that
   $\phi'(\ls)=a$, $\phi(\lambda)/\lambda$ is strictly convex on $(0,\Theta)$ and
   the sign of $\phi'(\lambda)-\phi(\lambda)/\lambda$ equals that of $\lambda-\lambda^*$.
 \end{lem}

 \begin{proof}
   Define $c(\lambda)=\phi(\lambda)/\lambda$ for $\lambda\in(0,\ls)$. A little calculus
   shows that $c(\lambda)$ is strictly convex:
   \[c''(\lambda)=\frac{\phi''(\lambda)}{\lambda}-\frac{2\phi'(\lambda)}{\lambda^2}
   +\frac{2\phi(\lambda)}{\lambda^3}
   =\frac{1}{\lambda^3}\int_{-\infty}^\infty\left[(\lambda x-1)^2+1)\right]e^{\lambda
     x}\rho(x)\,dx>0.\] It is clear that $c(\lambda)\to\infty$ as $\lambda\to0$. On the
   other hand, \eqref{eq:expMoments:2} implies $c(\lambda)\to\infty$ as
   $\lambda\to\Theta-$ as well.  Thus the minimum of $c$ is attained at some
   $\lambda^*\in(0,\Theta)$, and we have $c'(\ls)=0$, or
   $\phi'(\lambda^*)=\phi(\lambda^*)/\ls$, which will give the second claim in the lemma
   once we show that
   $c(\ls)=a$.
   
   Recalling the characterization of $a$ given after \eqref{eq:defLambda}, we
   need to show that
   \begin{equation}
     \sup_{\theta>0}\left[\theta c(\lambda^*)-\phi(\theta)\right]=0.
   \end{equation}
   This is easy: using the definition of $c$ we get
   \[\sup_{\theta>0}\left[\theta
     c(\lambda^*)-\phi(\theta)\right]\geq\lambda^*c(\lambda^*)-\phi(\lambda^*)
   =0,\]
   while for all $\theta>0$
   \[\theta c(\lambda^*)-\phi(\theta)\leq\theta c(\theta)-\phi(\theta)=0.\]

   Finally, to get the last claim in the lemma recall that $c'(\ls)=0$ and $c$ is convex,
   so 
   \[\phi'(\lambda)-\frac{\phi(\lambda)}{\lambda}=\lambda c'(\lambda)\]
   is negative for $\lambda\in(0,\ls)$ and positive for $\lambda\in(\ls,\Theta)$.
 \end{proof}

 Recall (see \eqref{eq:defK}) that
 \[k(x)=\frac{\lambda}{\phi(\lambda)}e^{\lambda x}R(x).\] As we explained in Section
 \ref{sec:wave} the proof of Theorem \ref{thm:wave} will depend on looking for positive
 solutions to the equation \eqref{eq:waveSol:densU}. We will actually consider a slightly
 more general equation:
 \begin{equation}
   \label{eq:waveSol:t}
   U(x)=\int_0^\infty\!U(y)k(x-y)\,dy\qquad\forall\,x\geq0,
 \end{equation}
where we look for a non-decreasing solution $U$, continuous except at the origin, with
$U(x)=0$ for all $x<0$.  
In Spitzer's terminology, a solution $U$ with these properties is a $P^*$-solution of
\eqref{eq:waveSol:t}. When $\lim_{x\to\infty}U(x)=1$ we call $U$ a $P$-solution, and
think of it as the distribution function of a non-negative random variable. The following summarizes
the two results of Spitzer that we will need.

\begin{thm2}[Theorems 2 and 4 in \citet{spitz1}]\label{thm:spitzer}
\mbox{}
\begin{enumerate}[label=(\alph*)]
\item If $\int_{-\infty}^{\infty}xk(x)\,dx\leq0$ then there is a
unique (up to multiplicative constant) $P^*$-solution of \eqref{eq:waveSol:t}.
\item If $\int_{-\infty}^{\infty}xk(x)\,dx<0$ then there is a unique $P$-solution of \eqref{eq:waveSol:t}
which can be obtained as the limit $U(x)=\lim_{n\to\infty}U_n(x)$ of the iterative procedure defined by
$U_{n+1}(x)=\int_0^\infty U_n(y)k(x-y)\,dy$ starting with an arbitrary continuous $U_0$
corresponding to the distribution function of a non-negative random variable.
\item If $\int_{-\infty}^{\infty}xk(x)\,dx\geq0$ then \eqref{eq:waveSol:t} has no $P$-solution.
\end{enumerate}
\end{thm2}

Repeating the arguments we used to show that $F(t,x)$ had a density (see \eqref{eq:nuPhi})
we see that if $U$ is a $P^*$-solution of \eqref{eq:waveSol:t} then there is a
non-negative function $u$ such that $U(x)=\int_0^x u(y)\,dy$. Again repeating previous
arguments (see the first part of the proof of Theorem \ref{thm:fb}), we deduce that $u$
satisfies \eqref{eq:waveSol:densU}, while obviously $u(x)=0$ for $x\leq0$. $u$ is
continuous except possibly at the origin by the dominated convergence theorem thanks to
the fact that $k$ is continuous.  Multiplying $u(x)$ by $e^{\lambda x}$ will allow us to
obtain a solution for \TW with the desired properties.

\begin{proof}[Proof of Theorem \ref{thm:wave}]
  Let $c\geq a$ and take $\lambda\in(0,\ls]$ such that $\phi(\lambda)/\lambda=c$ as above.
  The uniqueness of the solutions of \TW in this case follows from Theorem
  \ref{thm:spitzer}(a). In fact, if $w_1$ and $w_2$ are two solutions of
  \TW then $u_i(x)=e^{\lambda x}w_i(x)$ solves \eqref{eq:waveSol:densU} for
  $i=1,2$, and thus the functions $U_i(x)=\int_0^x u_i(y)\,dy$ are $P^*$-solutions of
  \eqref{eq:waveSol:t}, and they are continuous because the $u_i$ are locally
  integrable. Hence $U_1(x)=AU_2(x)$ for all $x\in\rr$ and some $A>0$, which implies that
  $w_1(x)=Aw_2(x)$ for all $x\in\rr$. Integrating this relation we get $A=1$ and
  uniqueness follows.
  
  To show existence we start by integrating by parts to obtain
  \[\int_{-\infty}^{\infty}\!xk(x)\,dx
  =\frac{\lambda}{\phi(\lambda)}\int_{-\infty}^{\infty}\!xe^{\lambda x}R(x)\,dx
  =\frac{1}{\phi(\lambda)}\int_{-\infty}^{\infty}\!
  \left(x-\frac{1}{\lambda}\right)e^{\lambda x}\rho(x)\,dx.\] Thus the sign of
  $\int_{-\infty}^{\infty}\!xk(x)\,dx$ is the same as that of
  $\phi'(\lambda)-\phi(\lambda)/\lambda$. This last quantity is strictly negative for
  $\lambda\in(0,\ls)$ and vanishes for $\lambda=\ls$ by Lemma \ref{lem:goodSpeed}.

  If $c>a$, then $\lambda<\ls$ and thus Theorem \ref{thm:spitzer}(b) provides us with a
  $P$-solution of \eqref{eq:waveSol:t} to which, by the discussion preceding this proof,
  we can associate a continuous (except at the origin) function $u$ satisfying
  \eqref{eq:waveSol:densU} and corresponding to the density of a non-negative random
  variable. Now let $A=\int_0^\infty e^{-\lambda x}u(x)\,dx$ which is clearly finite
  (actually $A<1$).  Define $w(x)=A^{-1}e^{-\lambda x}u(x)$. Then $w$ is the density of a
  non-negative random variable and it is easy to check that it satisfies
  \eqref{eq:waveSol:a} with $c=\phi(\lambda)/\lambda$. The random variable $w$ obviously
  satisfies \eqref{eq:waveSol:b}, \eqref{eq:waveSol:c} and \eqref{eq:waveSol:d} as well,
  it is continuous except possibly at the origin because so is $u$ and, by definition,
  $\int_0^\infty e^{\lambda x}w(x)\,dx=A^{-1}<\infty$. The last thing left to show in this
  case is that $w$ is differentiable (except at the origin), but this follows easily from
  writing, for $x\geq0$,
  \[w(x)=\frac{1}{c}\int_0^\infty\!w(y)\int_{x-y}^\infty\!\rho(z)\,dz\,dy
  =\frac{1}{c}\int_x^\infty\!\int_{0}^\infty\!w(y)\rho(z-y)\,dy\,dz,\]
  and using the fact that the integrand above is continuous. This establishes (a) in the
  case $c>a$.

  The equality $\int_0^\infty e^{\lambda x}w(x)\,dx=A^{-1}<\infty$ obtained above gives
  the first claim in (b). To prove the second claim in (b) we may obviously assume that
  $\widetilde\lambda\in(\lambda,\ls)$. Let $W(x)=\int_x^\infty w(y)\,dy$. It is not
  hard to check that $W$ satisfies
  \[W(x)=\frac{1}{c}\int_{-\infty}^\infty\!W(y)R(x-y)\,dy\quad\forall\,x\geq0.\]
  Then if $A=\sup_{x\geq0}e^{\widetilde\lambda x}W(x)<\infty$ we have that
  \begin{align}
    W(x)e^{\widetilde\lambda x}&=\frac{1}{c}\int_{-\infty}^{\infty}\!
    W(y)e^{\widetilde\lambda y}e^{\widetilde\lambda(x-y)}R(x-y)\,dy\\
    &\leq\frac{A}{c}\int_{-\infty}^{\infty}\!e^{\widetilde\lambda(x-y)}R(x-y)\,dy
    =\frac{A}{c}\frac{\phi(\widetilde\lambda)}{\widetilde\lambda}.
  \end{align}
  Taking supremum in $x\geq0$ and recalling that $c=\phi(\lambda)/\lambda$ the above says
  that
  \[A\leq\frac{\phi(\widetilde\lambda)}{\widetilde\lambda}\frac{\lambda}{\phi(\lambda)}A,\]
  and since $A>0$ this says that
  $\phi(\widetilde\lambda)/\widetilde\lambda\geq\phi(\lambda)/\lambda$.  But Lemma
  \ref{lem:goodSpeed} implies exactly the opposite for $\lambda<\widetilde\lambda<\ls$. This is a contradiction, and thus
  $A=\infty$, which finishes the proof of (b).

  The case $c=a$ is similar so we will skip some details. Now we have $\lambda=\ls$ and
  thus $\int_{-\infty}^{\infty}xk(x)\,dx=0$.  Theorem \ref{thm:spitzer}(a) provides us now
  with a $P^*$-solution of \TW to which corresponds a function $u$ which is
  the density of a measure supported on $[0,\infty)$ and which satisfies
  \eqref{eq:waveSol:densU}. Let $A=\int_0^\infty e^{-\ls x}u(x)\,dx$. We need this
  quantity to be finite, so that $w(x)=A^{-1}e^{-\ls x}u(x)$ is a continuous probability
  density. This follows from Theorem 6.2 of \citet{engib}, which assures that
  $U(x)=\int_0^xu(y)\,dy\leq Cx$ for some $C>0$, and integration by parts:
  \begin{align}
    A&=\int_0^\infty\!e^{-\ls x}u(x)\,dx=\lim_{x\to\infty}e^{-\ls
      x}U(x)-U(0)+\ls\int_0^\infty\!e^{-\ls x}U(x)\,dx\\
    &\leq C\ls\int_0^\infty\!e^{-\ls x}x\,dx<\infty.
  \end{align}
  It is easy again to verify that $w$ satisfies \TW, and its
  differentiability follows from the same reasons as above. Hence we have established (a)
  for the case $c=a$. Clearly
  \begin{equation}
    \int_0^x e^{\ls y}w(y)\,dy=A^{-1}\int_0^x u(y)\,dy=O(x)\label{eq:bdU}
  \end{equation}
  which is the second claim in (c). The first and third claims in (c) follow from two
  other consequences of the cited result in \citet{engib}, namely that $U(x)\to\infty$ as
  $x\to\infty$ and that $U$ is subadditive. For the third claim use the first of these
  properties of $U$ and the first equality in \eqref{eq:bdU}, while for the first one
  integrate by parts to get
  \begin{align}
    e^{\ls x}\int_x^\infty\!w(y)\,dy&=\int_x^\infty\!e^{-\ls(y-x)}u(y)\,dy
    =\int_0^\infty\!e^{-\ls z}u(x+z)\,dz\\
    &=\lim_{z\to\infty}e^{-\ls z}U(x+z)-U(x)+\int_0^\infty\!\ls e^{-\ls z}U(x+z)\,dz\\
    &\leq-U(x)+\int_0^\infty\!\ls e^{-\ls z}[U(x)+U(z)]\,dz
    \leq C\int_0^\infty\!\ls e^{-\ls z}z\,dz=\frac{C}{\ls}.
  \end{align}

  We are only left showing (d), that is, that there are no solutions of \TW
  when $c<a$. We start by observing that if $w$ were a solution then the above arguments
  would imply that $w$ is differentiable on the positive axis, and thus if we set
  $f_0(x)=w(x)$ in \FB we get $\gamma(t)=ct$. Therefore, to show the non-existence of
  solutions for $c<a$ it is enough to show that, given any $\ep>0$ and any $f_0$ supported
  on $[0,\infty)$, there is a $T>0$ such that the solution $(f(t,x),\gamma(t))$ of \FB satisfies
  \begin{equation}
    \label{eq:TGamma}
    \gamma(T)>(a-\ep)T.
  \end{equation}

  Recall the definition of the process $\hnu_t$ in the proof of Proposition
  \ref{prop:nuhat}, which corresponded to the deterministic measure valued limit of the
  branching random walk $\hnu^N_t$, and observe that we can run this process started with
  any initial measure (not necessarily an absolutely continuous one). Moreover,
  \eqref{eq:eqNuHat} still holds in this case by Theorem 5.3 of \citet{fourMel}. Consider
  a copy of $\hnu_t$ started with a unit mass at 0 and let $\widehat
  F(t,x)=\hnu_t([x,\infty))$. For $T>0$ we define $\chi_T>0$ to be such that $\widehat
  F(T,\chi_T)=1$.

  Applying \eqref{eq:eqNuHat} to $\varphi(y)=\uno{y\in[x,\infty)}$, we see that $\widehat
  F(t,x)$ satisfies \eqref{eq:FBp1} for all $x\in\rr$. Now consider a copy of $\nu_t$
  started at the product measure $\nu_0$ defined by $f_0$. Then $(F(t,x),\gamma(t))$ satisfies
  \FBp and, since $\gamma$ is strictly increasing, it satisfies \eqref{eq:FBp1} for
  $x=\gamma(T)$ and all $t\in[0,T)$. We deduce that
  \[\frac{\partial}{\partial t}\big(F(t,\gamma(T))-\widehat F(t,\gamma(T))\big)
  =\int_{-\infty}^\infty\!\big(F(t,y)-\widehat F(t,y)\big)\rho(\gamma(T)-y)\,dy\] for all
  $t\in[0,T)$. Since $F(0,x)=\int_x^\infty f_0(y)\,dy$ and $\widehat F(0,x)=\uno{x\leq0}$,
  we have that $F(0,x)\geq\widehat F(0,x)$ for all $x$ and thus the above equation implies
  that
  \[F(t,\gamma(T))\geq\widehat F(t,\gamma(T))\] for all $t\in[0,T)$.
  Now $\widehat F(t,x)$ is clearly continuous in $t$, while
  \begin{align}
    F(T,\gamma(T))-F(t,\gamma(T))&=F(t,\gamma(t))-F(t,\gamma(T))=\int_{\gamma(t)}^{\gamma(T)}
    \!f(t,y)\,dy\leq\int_{\gamma(t)}^{\gamma(T)}\!\hf(t,y)\,dy\\
    &\leq\int_{\gamma(t)}^{\gamma(T)}\!\hf(T,y)\,dy\xrightarrow[t\to T-]{}0
  \end{align}
thanks to the
  continuity of $\gamma$.  Therefore the inequality also holds for $t=T$, which gives
  $1=F(T,\gamma(T))\geq\widehat F(T,\gamma(T))$, whence
  \begin{equation}
    \chi_T\leq\gamma(T).\label{eq:bd1chi}
  \end{equation}

  To finish the proof of \eqref{eq:TGamma} we need to show that
    \begin{equation}\label{eq:bd2chi}
    \chi_T>(a-\ep)T
  \end{equation}
  for large enough $T$. Observe that $\hnu_t$ (which we recall is started with 
  $\hnu_0=\delta_0$), corresponds to the mean measure of the branching random walk
  $\xi^1_t$ (started with just one particle located at the origin). This can be made
  precise by writing down the formula for the generator of $\xi^1_t$ and applying
  it to functions of the form
  $\xi\mapsto\br{\xi}{\varphi}=\sum_{i=1}^{N(\xi)}\varphi(\xi(i))$, where $N(\xi)$ is the
  number of particles in the branching random walk configuration $\xi$, to deduce that
  after taking expectations the resulting equation is the same as \eqref{eq:eqNuHat}. We
  leave the details to the reader, and instead only state that the above implies that
  \begin{equation}
    \hnu_T([cT,\infty))=\ee\!\left(\xi^1_T([cT,\infty))\right)\label{eq:hnuBRW}
  \end{equation}
  for every $c$, where $\xi^1_T([cT,\infty))$ denotes the number of particles in the
  branching random walk to the right of $cT$ at time $T$. On the other hand it is well
  known that
  \[\ee\!\left(\xi^1_T([cT,\infty))\right)=e^T\pp\!\left(S_T\geq cT\right)\]
  where $S_t$ is defined as in \eqref{eq:defLambda} (see for instance the third equation
  in the proof of Proposition I.1.21 in \citet{ligg2}), and thus \eqref{eq:defLambda}
  implies that
  \[\lim_{T\to\infty}\frac{1}{T}\log\!\left(\ee\!\left(\xi^1_T([cT,\infty))\right)\right)=\Lambda(c)+1.\]
  Since $\Lambda$ is strictly decreasing on $[0,\infty)$ and $\Lambda(a)=-1$ we deduce
  that
  \[\lim_{T\to\infty}\frac{1}{T}\log\!\left(\ee\!\left(\xi^1_T([cT,\infty))\right)\right)>0\qquad\text{for
    all }0\leq c<a.\]
  This together with \eqref{eq:hnuBRW} implies that
  \begin{equation}
    \hnu_T([(a-\ep)T,\infty))>1\label{eq:aep}
  \end{equation}
  for  large enough $T$. Therefore \eqref{eq:bd2chi} holds and the proof is complete.
\end{proof}

\paragraph{Acknowledgments}
The authors would like to thank Nathana\"{e}l Berestycki and Lee Zhuo Zhao for pointing out
a mistake in an earlier version of one of the proofs, and an anonymous referee
for valuable comments and suggestions.\\[2pt]
\noindent
\emph{(Added April 2023)} The authors also thank Rami Atar for pointing out a mistake in the proof of uniqueness of solutions to \FB in the published version of this paper.

\bibliographystyle{natbib} \bibliography{biblio}

\begin{thebibliography}{}

\bibitem[Atar(2023)Atar]{atar}
Atar, R. (2023).
\newblock Hydrodynamics of particle systems with selection via uniqueness for
  free boundary problems.
\newblock arXiv:2011.07535.

\bibitem[Athreya and Ney(1978)Athreya and Ney]{athNey}
Athreya, K.~B. and Ney, P. (1978).
\newblock A new approach to the limit theory of recurrent {M}arkov chains.
\newblock {\em Trans. Amer. Math. Soc.}, {\bf 245}, 493--501.

\bibitem[B\'erard and Gou\'er\'e(2010)B\'erard and Gou\'er\'e]{berGou}
B\'erard, J. and Gou\'er\'e, J.-B. (2010).
\newblock Brunet-derrida behavior of branching-selection particle systems on
  the line.
\newblock \emph{To appear in Comm. Math. Phys.} arXiv:0811.2782.

\bibitem[Biggins(1977)Biggins]{bigg}
Biggins, J.~D. (1977).
\newblock Chernoff's theorem in the branching random walk.
\newblock {\em J. Appl. Probability\/}, {\bf 14}(3), 630--636.

\bibitem[Brunet and Derrida(1997)Brunet and Derrida]{brunDerr}
Brunet, E. and Derrida, B. (1997).
\newblock Shift in the velocity of a front due to a cutoff.
\newblock {\em Phys. Rev. E (3)\/}, {\bf 56}(3, part A), 2597--2604.

\bibitem[Brunet {\em et~al.}(2007)Brunet, Derrida, Mueller, and
  Munier]{brunDerr2}
Brunet, {\'E}., Derrida, B., Mueller, A.~H., and Munier, S. (2007).
\newblock Effect of selection on ancestry: an exactly soluble case and its
  phenomenological generalization.
\newblock {\em Phys. Rev. E (3)\/}, {\bf 76}(4), 041104, 20.

\bibitem[Chayes and Swindle(1996)Chayes and Swindle]{chaySwin}
Chayes, L. and Swindle, G. (1996).
\newblock Hydrodynamic limits for one-dimensional particle systems with moving
  boundaries.
\newblock {\em Ann. Probab.}, {\bf 24}(2), 559--598.

\bibitem[Durrett(2004)Durrett]{durrPTE}
Durrett, R. (2004).
\newblock {\em Probability: theory and examples\/}.
\newblock Duxbury Press, Belmont, CA, third edition.

\bibitem[Durrett and Mayberry(2010)Durrett and Mayberry]{durrMay}
Durrett, R. and Mayberry, J. (2010).
\newblock Evolution in predator-prey systems.
\newblock {\em Stoch. Proc. Appl.}

\bibitem[Engibaryan(1996)Engibaryan]{engib}
Engibaryan, N.~B. (1996).
\newblock Convolution equations containing singular probability distributions.
\newblock {\em Izv. Ross. Akad. Nauk Ser. Mat.}, {\bf 60}(2), 21--48.

\bibitem[Feller(1971)Feller]{feller}
Feller, W. (1971).
\newblock {\em An introduction to probability theory and its applications.
  {V}ol. {II}.}
\newblock Second edition. John Wiley \& Sons Inc., New York.

\bibitem[Fournier and M{\'e}l{\'e}ard(2004)Fournier and
  M{\'e}l{\'e}ard]{fourMel}
Fournier, N. and M{\'e}l{\'e}ard, S. (2004).
\newblock A microscopic probabilistic description of a locally regulated
  population and macroscopic approximations.
\newblock {\em Ann. Appl. Probab.}, {\bf 14}(4), 1880--1919.

\bibitem[Gravner and Quastel(2000)Gravner and Quastel]{gravQuas}
Gravner, J. and Quastel, J. (2000).
\newblock Internal {DLA} and the {S}tefan problem.
\newblock {\em Ann. Probab.}, {\bf 28}(4), 1528--1562.

\bibitem[Joffe and M{\'e}tivier(1986)Joffe and M{\'e}tivier]{joffeMetiv}
Joffe, A. and M{\'e}tivier, M. (1986).
\newblock Weak convergence of sequences of semimartingales with applications to
  multitype branching processes.
\newblock {\em Adv. in Appl. Probab.}, {\bf 18}(1), 20--65.

\bibitem[Kre{\u\i}n(1962)Kre{\u\i}n]{krein}
Kre{\u\i}n, M.~G. (1962).
\newblock Integral equations on a half-line with kernel depending upon the
  difference of the arguments.
\newblock {\em Transl. Amer. Math. Soc. (2)\/}, {\bf 22}, 163--288.
\newblock Original (Russian): \emph{Uspekhi Mat. Nauk.}, {\bf 13}(5) (1958),
  3-120.

\bibitem[Landim {\em et~al.}(1998)Landim, Olla, and Volchan]{landimOllaVolchan}
Landim, C., Olla, S., and Volchan, S.~B. (1998).
\newblock Driven tracer particle in one-dimensional symmetric simple exclusion.
\newblock {\em Comm. Math. Phys.}, {\bf 192}(2), 287--307.

\bibitem[Liggett(1999)Liggett]{ligg2}
Liggett, T.~M. (1999).
\newblock {\em Stochastic interacting systems: contact, voter and exclusion
  processes\/}, volume 324.
\newblock Springer-Verlag, Berlin.

\bibitem[Meirmanov(1992)Meirmanov]{meirmanov}
Meirmanov, A.~M. (1992).
\newblock {\em The {S}tefan problem\/}, volume~3 of {\em de Gruyter Expositions
  in Mathematics\/}.
\newblock Walter de Gruyter \& Co., Berlin.

\bibitem[Paley and Wiener(1987)Paley and Wiener]{paleyWien}
Paley, R. E. A.~C. and Wiener, N. (1987).
\newblock {\em Fourier transforms in the complex domain\/}, volume~19 of {\em
  AMS Colloquium Publications\/}.
\newblock American Mathematical Society.

\bibitem[Roelly-Coppoletta(1986)Roelly-Coppoletta]{roel}
Roelly-Coppoletta, S. (1986).
\newblock A criterion of convergence of measure-valued processes: application
  to measure branching processes.
\newblock {\em Stochastics\/}, {\bf 17}(1-2), 43--65.

\bibitem[Spitzer(1957)Spitzer]{spitz1}
Spitzer, F. (1957).
\newblock The {W}iener-{H}opf equation whose kernel is a probability density.
\newblock {\em Duke Math. J.}, {\bf 24}, 327--343.

\bibitem[Wiener and Hopf(1931)Wiener and Hopf]{wienHopf}
Wiener, N. and Hopf, E. (1931).
\newblock \"{U}eber eine klasse singul\"arer integralgleichungen.
\newblock {\em Sitzungber. Akad. Wiss. Berlin\/}, page 696706.

\end{thebibliography}

\end{document}